\documentclass[10pt]{article}
   \usepackage{graphicx}
   \usepackage{epsfig}
   \usepackage{amssymb}
   \usepackage{amsmath}
   \usepackage{amsthm}
   \newtheorem{lemma}{Lemma}[section]
   \newtheorem{theorem}{Theorem}[section]

   {\end{description}}

   {\hfill $\bullet$ \\}

   \newcommand{\be}{\begin{equation}}
   \newcommand{\ee}{\end{equation}}

\begin{document}
    \title{An explicit computational approach for a three-dimensional system of nonlinear elastodynamic sine-Gordon problem}
   \author{Eric Ngondiep \thanks{\textbf{Email address:} ericngondiep@gmail.com}}
   \date{\small{Department of Mathematics and Statistics, College of Science, Imam Mohammad Ibn Saud\\ Islamic University (IMSIU), $90950$ Riyadh $11632,$ Saudi Arabia.}}
   \maketitle

   \textbf{Abstract.}
   This paper proposes an explicit computational method for solving a three-dimensional system of nonlinear elastodynamic sine-Gordon equations subject to appropriate initial and boundary conditions. The time derivative is approximated by interpolation technique whereas the finite element approach is used to approximate the space derivatives. The developed numerical scheme is so-called, high-order explicit computational technique. The new algorithm efficiently treats the time derivative term and provides a suitable time step restriction for stability and convergence. Under this time step limitation, both stability and error estimates of the proposed approach are deeply analyzed using a constructed strong norm. The theoretical studies indicate that the developed approach is temporal second-order convergent and spatially third-order accurate. Some numerical examples are carried out to confirm the theory, to validate the computational efficiency and to demonstrate the practical applicability of the new computational technique.\\
    \text{\,}\\

   \\
   {\bf AMS Subject Classification (MSC): 65M12, 65M06, 74H15}.

  \section{Introduction}\label{sec1}

   \text{\,\,\,\,\,\,\,\,\,\,}The elastodynamic model is mainly based on Newtow's second law of motion in homogeneous equivalent elastic media or Hamilton's principle in various complex media \cite{9zzllw,6zzllw,14zzllw,12zzllw}. These subsurface media are formed of pore fluids and solid skeletons such as underground nuclear tests, gas, oil, quick-frozen meals, water, permafrost in the arctic area, petroleum, cement material and sandstones, etc... Wave propagation analysis in porous media is crucial for monitoring the hydrological environment, predicting earthquakes, exploring oil and gas, monitoring frozen food, and characterizing sandstone reservoirs. However, earthquake and underground reservoir media are primary causes of nonlinear deformation (Figure $\ref{fig1}$). In the literature \cite{22zzllw,25zzllw,27zzllw,31zzllw,33zzllw, 28zzllw}, several authors formulated elastodynamic equations of motion in such media using the principle of least action or Hamilton's principle such as for porous media and piezoelectric materials. These equations generally model a three-dimensional system of tectonic deformation \cite{2enarxiv,33enarxiv,34enarxiv,enarxiv,40enarxiv} and can be used to describe the propagation of wave in anisotropic and multiphasic porous elastic media. Moreover, the three dimensional system of nonlinear elastodynamic sine-Gordon equations may be used for simulating a geological structure deformation case (for example, see Figure $\ref{fig1}$) caused by underground nuclear test exploration or earthquake occurring at a focus $B(\bar{x},r_{0})$ of the porous media \cite{2jl}. Because systems of nonlinear tectonic deformation equations deal with a nonlinear source term, they fall in the class of complex ordinary/partial differential equations (ODEs/PDEs) which are too difficult and sometimes impossible to solve analytically \cite{1en,10en,4lgzc,2en,3en,8lgzc,4en,5en,16lgzc,6en,7en,17lgzc,8en}.\\

        \begin{figure}
         \begin{center}
         \begin{tabular}{c c}
         \psfig{file=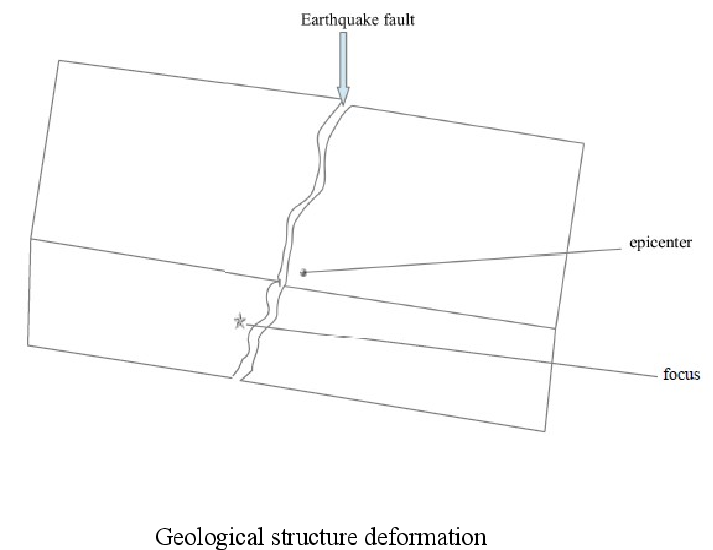,width=8cm} & \psfig{file=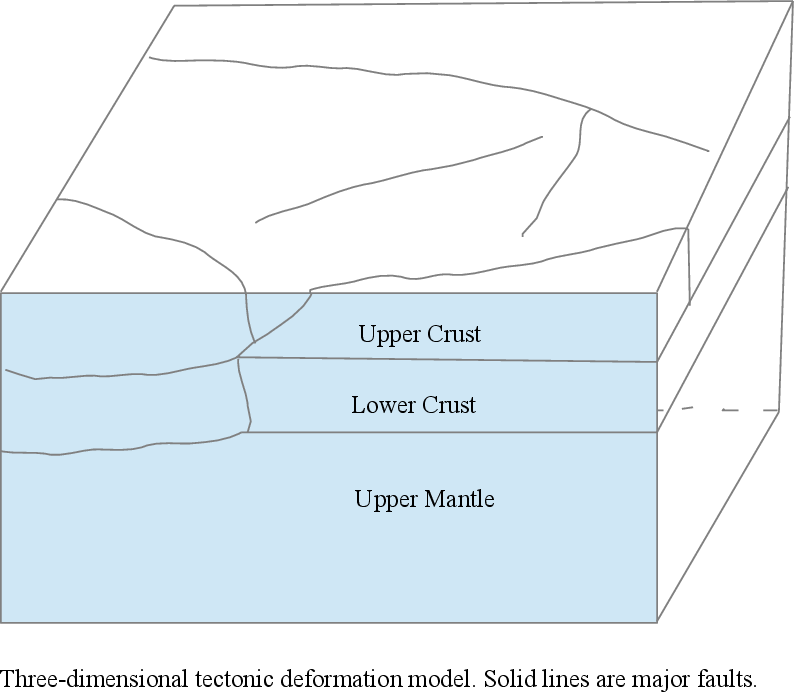,width=8cm}\\
         \psfig{file=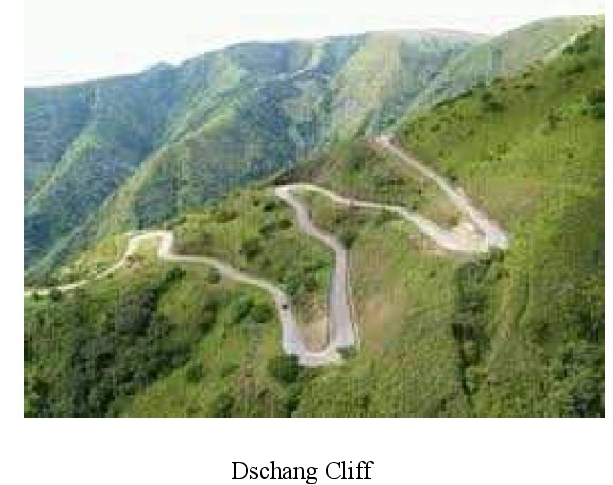,width=8cm} & \psfig{file=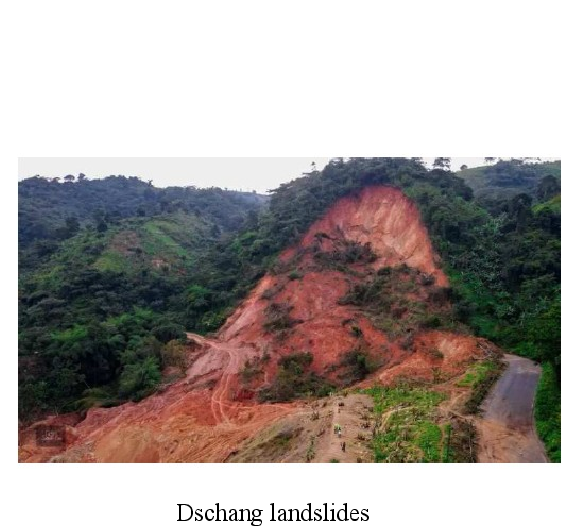,width=8cm}\\
         \end{tabular}
        \end{center}
        \caption{Geological structure deformation.}
        \label{fig1}
        \end{figure}

     In this paper, we consider the following three-dimensional system of nonlinear elastodynamic sine-Gordon equations defined in \cite{jtfyl} as
      \begin{equation}\label{1}
     \left\{
      \begin{array}{ll}
        \rho \frac{\partial^{2}u}{\partial t^{2}}-(\lambda_{1}+\lambda_{2})\nabla(\nabla\cdot u)-\lambda_{2}\nabla\cdot\overline{\nabla}u=F(u), & \hbox{on $\Omega\times[0,\text{\,}T]$}\\
        \text{\,}\\
        \psi-\lambda_{1}(\nabla\cdot u)\mathcal{I}-\lambda_{2}(\nabla u+(\nabla u)^{T})=0, & \hbox{on $\Omega\times[0,\text{\,}T]$}
      \end{array}
    \right.
    \end{equation}
    subjects to initial conditions
    \begin{equation}\label{2}
    u(x,0)=u_{0},\text{\,\,\,\,\,}u_{t}(x,0)=v(x),\text{\,\,\,\,\,}\psi(x,0)=\psi^{0}(x)=\lambda_{1}(\nabla\cdot u_{0}(x))\mathcal{I}+\lambda_{2}(\nabla u_{0}(x)+(\nabla u_{0}(x))^{T}),\text{\,\,\,\,\,\,\,\,\,\,\,\,on\,\,\,\,}\Omega,
    \end{equation}
    and boundary conditions
    \begin{equation}\label{3}
    u(x,t)=0,\text{\,\,\,\,\,}\psi(x,t)=0,\text{\,\,\,\,\,\,\,\,\,\,\,\,\,\,\,\,\,\,\,\,\,\,\,\,\,\,on\,\,\,\,}\Gamma\times[0,\text{\,}T],
    \end{equation}
    where $\Omega\subset\mathbb{R}^{3}$ denotes a bounded and connected domain, $\Gamma$ is the boundary of $\Omega$, $T$ designates the final time, $u(x,t)$ stands for the displacement at time $t$ of the material particle located at position $x\in\Omega$ while $\psi$ means the symmetric stress tensor. In addition, $\Delta$, $\nabla$, $\overline{\nabla}$ and $\nabla\cdot$, represent the laplacian, gradient, Jacobian and divergence operators, respectively, whereas $\overline{\nabla}w$ designates the jacobian matrix of a vector $u=(u_{1},u_{2},u_{3})^T\in\mathbb{R}^{3}$, $\nabla u=(\nabla u_{1}, \nabla u_{2},\nabla u_{3})^{T}$ and $\nabla\cdot\nabla u=(\nabla\cdot\nabla u_{1},\nabla\cdot\nabla u_{2},\nabla\cdot\nabla u_{3})^{T}$, where $w^{T}$ means the transpose of the vector $w$. $\rho$ is called the elastic body density while $\lambda_{1}$ and $\lambda_{2}$ are two physical parameters which may be defined as: $\lambda_{1}=\frac{\beta E}{(1+\beta)(1-2\beta)}$, $\lambda_{2}=\frac{E}{2(1+\beta)}$, where $E$ and $\beta$ denote elastic modulus and poisson's ratio, respectively. Finally, $\mathcal{I}$ is the identity operator, $F(u)$ is the nonlinear source term, $\psi^{0}$ denotes the initial stress tensor whereas $u_{0}$ and $u_{1}$ are initial conditions defined by
   \begin{equation}\label{4}
    u_{0}(x)=v(x)=\left\{
             \begin{array}{ll}
               \sin(2\pi r_{0}|x-\bar{x}|)\hat{x}, & \hbox{for $x\in B(\bar{x},r_{0})$} \\
               \text{\,}\\
               0, & \hbox{for $x\in\Omega\cup\Gamma\setminus B(\bar{x},r_{0})$}
             \end{array}
           \right.
   \end{equation}
   
   Here: $\hat{x}=(1,1,1)^{T}$, $\bar{x}\in\Omega$ is a fixed point, $r_{0}>0$, $|\cdot|$ is the usual norm in the vector space $\mathbb{R}^{3}$, $B(\bar{x},r_{0})=\{x\in\Omega,\text{\,}|x-\bar{x}|<r_{0}\}$ is the opened ball centered at the point $\bar{x}$ with radius $r_{0}$ and $\Gamma_{B(\bar{x},r_{0})}$ designates its boundary. The nonlinear term $F(u)$ is given by $F(u)=(\sin(u_{1}),\sin(u_{2}),\sin(u_{3}))^{T}$. Since $|x-\bar{x}|=r_{0}$ for every $x\in\Gamma_{B(\bar{x},r_{0})}$, so $\sin(2\pi r_{0}|x-\bar{x}|)=\sin(2\pi r_{0}^{2})$, on $\Gamma_{B(\bar{x},r_{0})}$. For $r_{0}=\sqrt{\frac{p}{2}}$, where $p\in\mathbb{N}\setminus\{0\}$, $u_{0}$ and its spatial derivatives vanish on $\Gamma_{B(\bar{x},r_{0})}$, this suggests that $u_{0}$ has continuous partial derivatives at any order defined on $\overline{\Omega}$. This fact along with the first equation in system $(\ref{1})$ show that $F(u)\in[H^{3}(0,T;\text{\,}H^{3}]^3$ which implies that the exact solution $u$ falls in the space $[H^{5}(0,T;\text{\,}H^{5}]^3$.\\

     In the literature, a broad range of numerical methods have been developed for solving time-dependent nonlinear PDEs. Such schemes include weak Galerkin finite element methods \cite{26enarxiv,8enarxiv,3enarxiv,41enarxiv,38enarxiv,33enarxiv}, higher-order finite difference/finite element schemes \cite{40enarxiv,2enarxiv,enarxiv,17en,16enarxiv,24enarxiv,48dg,22enarxiv,37enarxiv,19enarxiv,9en,18enarxiv,34enarxiv,28enarxiv} and spectral finite element techniques \cite{42dg,14dg}. Although several of the approaches mentioned above are simple and time-saving, they are limited to cartesian grids, which are not appropriate for geometrical internal or surface nonlinearities. However, these limitations have been solved by the use of discontinuity grids or nonuniform ones at the free surface \cite{6dg,54dg} and subdomains dealing with suited grids to complicated geometries have been established and deeply studied \cite{48dg,8dg}. Unfortunately, they are extremely expensive because they need the inversion of fully block matrices at each time level. In this paper, we construct an efficient explicit computational approach in an approximate solution of the three-dimensional system of nonlinear elastodynamic sine-Gordon equations $(\ref{1})$ with initial-boundary conditions $(\ref{2})$-$(\ref{3})$. The developed procedure is a combination of interpolation technique with the finite element formulation. More precisely, the time derivative is estimated by second-degree polynomial interpolations, while the space derivatives are approximated by third-degree polynomials defined locally on each mesh element. Under a suitable time step requirement, known as the Courant-Friedrich-Lewy (CFL) condition for the stability of explicit numerical methods applied to linear hyperbolic PDEs, the new algorithm is stable, second-order accurate in time, and spatially third-order convergent. Because it is explicit and uses the finite element formulation, there is no need to invert a block matrix at each time step, but it does allow for adaptive space sizes for geometrical internal or surface nonlinearities. Additionally, the proposed numerical scheme is appropriate for solving dynamic equations in multi-physics coupling systems and may be used to simulate wave propagation issues with complex boundaries, including reflections and refractions at interfaces. This fact, together with its high-order accuracy suggest that the proposed computational technique is more efficient than a broad range of numerical schemes \cite{8dg,dg,14dg,48dg,54dg} deeply analyzed in the literature in computed solutions of hyperbolic systems of elastodynamic problems. It's worth mentioning that the highlights of this study considers the following three items.\\
   \begin{description}
     \item[1.] Construction of the high-order explicit computational approach for simulating the initial-boundary value problem $(\ref{1})$-$(\ref{3})$.
     \item[2.] Analysis of stability and error estimates of the developed numerical technique.
     \item[3.] Numerical examples to confirm the theoretical results and validate the accuracy and efficiency of the proposed algorithm.
   \end{description}

   The remainder of the paper is organized as follows. Section $\ref{sec2}$ considers a detailed description of the combined interpolation with finite element method for solving a three-dimensional system of nonlinear elastodynamic sine-Gordon equations $(\ref{1})$, subject to initial-boundary conditions $(\ref{2})$-$(\ref{3})$. In Section $\ref{sec3}$, we analyze under an appropriate time step restriction both stability and error estimates of the constructed computational method. Some numerical examples are carried out in Section $\ref{sec4}$ to demonstrate the utility and validity of the new algorithm while the general conclusions and future works are provided in Section $\ref{sec5}$.\\

     \section{Full description of the computational approach}\label{sec2}
     
     \text{\,\,\,\,\,\,\,\,\,\,}This section develops an explicit computational approach in an approximate solution of a three-dimensional system of nonlinear elastodynamic sine-Gordon equations $(\ref{1})$ subject to initial-boundary conditions $(\ref{2})$-$(\ref{3})$. The new algorithm derives from a combined interpolation technique with finite element formulation.\\

      Let $N>0$, be an integer and $\sigma=\frac{T}{N}$ be the time step. Set $t_{n}=n\sigma$, for $n=0,1,2,...,N$, be the discrete time and $\tau_{\sigma}=\{t_{n}=n\sigma,\text{\,\,}0\leq n\leq N\}$ be the uniform partition of the interval $[0,\text{\,}T]$. Consider $\Pi_{h}$ be a finite element method (FEM) triangulation of the domain $\overline{\Omega}=\Omega\cup\Gamma$, which consists of tetrahedra $T$, with maximum diameter $"h"$. Moreover, $h$ denotes the step size of a spatial mesh of the computational domain $\Omega\cup\Gamma$, while the finite element triangulation $\Pi_{h}$ satisfies the following assumptions: (a) the intersection of two tetrahedra is either a common face/edge or the empty set; (b) the intersection of the interior of two different tetrahedra is the empty set; (c) the interior of any tetrahedron is nonempty; (d) the triangulation $\Pi_{h}$ is uniform and the triangulation $\Pi_{\Gamma,h}$ induced on the boundary $\Gamma=\partial\Omega$ is quasi-uniform.\\

     We introduce the Sobolev spaces:
     \begin{equation*}
     \mathcal{U}=\{u=(u_{1},u_{2},u_{3})^{T}\in[H^{1}(\Omega)]^{3}:\text{\,}u|_{\Gamma}=0\};\text{\,}
     \mathcal{M}=\{\tau=(\tau_{ij})\in\mathcal{M}_{3}(L^{2}(\Omega)):\text{\,}\tau_{ij}=\tau_{ji},\text{\,}\tau_{ij}|_{\Gamma}=0\};
     \end{equation*}
     \begin{equation}\label{5}
     \mathcal{P}=\mathcal{U}\times\mathcal{M};\text{\,\,\,} W_{2}^{5}(0,T;\text{\,}H^{5})=\{v\in L^{2}(0,T;\text{\,}H^{5}):\text{\,}\frac{\partial^{m}v}{\partial t^{m}}\in L^{2}(0,T;\text{\,}H^{5}),\text{\,\,for\,\,}m=1,2,...,5\},
      \end{equation}
      where $u^{T}$ means the transpose  of the vector $u$, $\mathcal{M}_{3}(L^{2}(\Omega))$ designates the space of $3\times3$-matrices with elements in $L^{2}(\Omega)$.\\

     We equip the spaces $L^{2}(\Omega)$ and $[L^{2}(\Omega)]^{3}$ with the inner products $\left(\cdot,\cdot\right)_{0}$ and $\left(\cdot,\cdot\right)_{\bar{0}}$, and norms $\|\cdot\|_{0}$ and $\|\cdot\|_{\bar{0}}$, respectively, while the Sobolev spaces $H^{5}(\Omega)$ and $\mathcal{U}$ are equipped with the norms $\|\cdot\|_{5}$ and $\|\cdot\|_{\bar{1}}$, respectively. The above scalar products and norms are defined by:
      \begin{equation*}
     \left(w,v\right)_{0}=\int_{\Omega}wv d\Omega,\text{\,\,}\|w\|_{0}=\sqrt{\left(w,w\right)_{0}},\text{\,\,}\forall w,v\in L^{2}(\Omega);\text{\,\,\,}
     \left(u,z\right)_{\bar{0}}=\int_{\Omega}u^{T}z d\Omega,\text{\,\,}\|z\|_{\bar{0}}=\sqrt{\underset{i=1}{\overset{3}\sum}\|z_{i}\|_{0}^{2}},
     \end{equation*}
      \begin{equation}\label{6}
       \text{\,\,for\,\,} u=(u_{1},u_{2},u_{3})^{T},z=(z_{1},z_{2},z_{3})^{T}\in [L^{2}(\Omega)]^{3}; \text{\,\,}
       \|v\|_{1}=\sqrt{\|v\|_{0}^{2}+\|\nabla v\|_{\bar{0}}^{2}},\text{\,\,}\forall v\in W_{2}^{1}(\Omega).
     \end{equation}

      For $m=(m_{1},m_{2},m_{3})\in\mathbb{N}^{3}$ and $x=(x_{1},x_{2},x_{3})\in\mathbb{R}^{3}$, $D_{x}^{m}v=\frac{\partial^{|m|}v}{\partial x^{|m|}}$ and $D_{x}^{0}v:=v$, where  $|m|=\underset{l=1}{\overset{3}\sum}m_{l}$ and $\partial x^{|m|}=\partial x_{1}^{m_{1}}\partial x_{2}^{m_{2}}\partial x_{3}^{m_{3}}$. Thus
       \begin{equation}\label{7}
       \|v\|_{5}=\sqrt{\underset{|m|=0}{\overset{5}\sum}\|D_{x}^{m}v\|_{0}^{2}},\text{\,\,}\forall v\in H^{5}(\Omega)\text{\,\,and\,\,}
       \|z\|_{\bar{1}}=\sqrt{\underset{l=1}{\overset{3}\sum}\left\|z_{l}\right\|_{1}^{2}},\text{\,\,for\,\,}z=(z_{1},z_{2},z_{3})^{T}\in\mathcal{U}.
      \end{equation}

       Furthermore, we endow the space of stress tensors $\mathcal{M}$ with the inner product $\left(\cdot,\cdot\right)_{*}$ and norm $\|\cdot\|_{*}$, whereas the spaces $W_{2}^{5}(0,T;\text{\,}H^{5})$ and $\mathcal{P}=\mathcal{U}\times\mathcal{M}$ are equipped with the norms $\||\cdot|\|_{5,5}$ (also $\||\cdot|\|_{5,\infty}$) and $\||\cdot|\|_{\bar{1},*}$, respectively.
      \begin{equation*}
        \left(\tau_{1},\tau_{2}\right)_{*}=\int_{\Omega}\underset{i=1}{\overset{3}\sum}\underset{j=1}{\overset{3}\sum}(\tau_{1})_{ij}(\tau_{2})_{ij}d\Omega;\text{\,\,\,} \|\tau_{1}\|_{*}=\sqrt{\left(\tau_{1},\tau_{1}\right)_{*}},\text{\,\,for\,\,}\tau_{1}=(\tau_{1})_{ij},\tau_{2}=(\tau_{2})_{ij}\in\mathcal{M};
     \end{equation*}
      \begin{equation*}
      \||v|\|_{5,5}=\sqrt{\underset{l=0}{\overset{5}\sum}\int_{0}^{T}\left\|\frac{\partial^{l}v(t)}{\partial t^{l}}\right\|_{5}^{2}dt},\text{\,\,\,} \||v|\|_{5,\infty}=\underset{0\leq t\leq T}{\max}\|v(t)\|_{5},\text{\,\,\,}\text{\,\,}v\in W_{2}^{5}(0,T;\text{\,}H^{5});
      \end{equation*}
      \begin{equation}\label{8}
      \||(u,\tau)|\|_{\bar{1},*}=\sqrt{\|u\|_{\bar{1}}^{2}+\|\tau\|_{*}^{2}},\text{\,\,\,\,\,} (u,\tau)\in\mathcal{P}.
      \end{equation}

      Let $\mathcal{U}_{h}$ and $\mathcal{M}_{h}$ be the finite element spaces approximating the solution of the initial-boundary value problem $(\ref{1})$-$(\ref{3})$.
      \begin{equation}\label{9}
      \mathcal{U}_{h}=\{u_{h}(t)\in\mathcal{U},\text{\,}u_{h}(t)|_{T}\in[\mathcal{Q}_{5}(T)]^{3},\text{\,}\forall t\in[0,\text{\,}T],\text{\,\,}\forall T\in\Pi_{h}\},
      \end{equation}
      where $\mathcal{Q}_{5}(T)$ represents the set of all polynomials defined on $T$ with degree less than or equal $5$;
      \begin{equation}\label{10}
      \mathcal{M}_{h}=\{\tau_{h}(t):=\tau(u_{h}(t))\in\mathcal{M},\text{\,}u_{h}(t)\in\mathcal{U}_{h},\text{\,\,for\,\,}0\leq t\leq T\}.
      \end{equation}

      Finally, consider the bilinear operator $B(\cdot,\cdot)$ defined as
      \begin{equation}\label{11}
      B(u,v)=(\lambda_{1}+\lambda_{2})\left(\overline{\nabla}u,\overline{\nabla}v\right)_{*}+\lambda_{2}\left(\nabla\cdot u,\nabla\cdot v\right)_{0}, \text{\,\,\,for\,\,\,}u,v\in\mathcal{U}.
      \end{equation}
      
      We remind that $\overline{\nabla}w$ and $\nabla\cdot w$ are the Jacobian matrix and the divergence of the vector $w$. For a vector $v\in \mathcal{U}$ and a stress tensor $\tau\in\mathcal{M}(L^{2}(\Omega))$, the Green formula is given by
      \begin{equation}\label{12}
      \left(\nabla\cdot\tau,v\right)_{\bar{0}}=-\left(\tau,\overline{\nabla}v\right)_{*}+\int_{\Gamma}(\tau\vec{z})^{t}vd\Gamma,
      \end{equation}
      where $\vec{z}$ designates the unit outward normal vector on $\Gamma$, $w^{T}$ is the transpose of a vector $w$, $"\nabla\cdot"$ denotes the divergence operator and $\overline{\nabla}v$ is the Jacobian matrix of the vector $v$. The operators $\nabla\cdot$ and $\overline{\nabla}$ are defined as
      \begin{equation}\label{13}
      (\nabla\cdot\tau)_{j}=\underset{i=1}{\overset{3}\sum}\frac{\partial\tau_{ij}}{\partial x_{i}}\text{\,\,\,\,and\,\,\,}\overline{\nabla}v=\left(\frac{\partial v_{i}}{\partial x_{j}}\right)_{ij},\text{\,\,\,}1\leq i,j\leq3.
      \end{equation}

      Since the formulas can become quite heavy, for the convenience of writing, we set $u:=u(t)\in\mathcal{U}$ and $\tau:=\tau(t)\in\mathcal{M}$, $\forall t\in[0,\text{\,}T]$.      Integrating the first equation in system $(\ref{1})$ on the interval $[t_{n-1},\text{\,}t_{n+1}]$ and rearranging terms, this provides
      \begin{equation*}
      \rho(u^{n+1}_{t}-u^{n-1}_{t})=\int_{t_{n-1}}^{t_{n+1}}[(\lambda_{1}+\lambda_{2})\nabla(\nabla\cdot u)+\lambda_{2}\nabla\cdot\overline{\nabla}u+F(u)]dt.
      \end{equation*}
      
      Multiplying both sides of this equation by $w\in\mathcal{U}$ and using the scalar product $(\cdot,\cdot)_{\bar{0}}$ defined in relation $(\ref{6})$, we obtain
      \begin{equation*}
      \rho\left(u^{n+1}_{t}-u^{n-1}_{t},w\right)_{\bar{0}}=\left(\int_{t_{n-1}}^{t_{n+1}}[(\lambda_{1}+\lambda_{2})\nabla(\nabla\cdot u)+\lambda_{2}\nabla\cdot\overline{\nabla}u+F(u)]dt,w\right)_{\bar{0}},
      \end{equation*}
      which is equivalent to
      \begin{equation}\label{14}
      \left(u^{n+1}_{t}-u^{n-1}_{t},w\right)_{\bar{0}}=\frac{1}{\rho}\int_{t_{n-1}}^{t_{n+1}}\left[(\lambda_{1}+\lambda_{2})\left(\nabla(\nabla\cdot u),w\right)_{\bar{0}}+\lambda_{2}\left(\nabla\cdot\overline{\nabla}u,w\right)_{\bar{0}}+\left(F(u),w\right)_{\bar{0}}\right]dt.
      \end{equation}

      Since $u,w\in\mathcal{U}$, so $u|_{\Gamma}=w|_{\Gamma}=0$. Thus, the application of the Green formula $(\ref{12})$ gives
      \begin{equation}\label{15}
      \left(\nabla\cdot\overline{\nabla}u,w\right)_{\bar{0}}=-\left(\overline{\nabla}u,\overline{\nabla}w\right)_{*}+\int_{\Gamma}(\overline{\nabla}u\vec{z})^{t}wd\Gamma=
      -\left(\overline{\nabla}u,\overline{\nabla}w\right)_{*}.
      \end{equation}
      
      But it has been established in \cite{enarxiv}, page 10, that
      \begin{equation}\label{16}
      \left(\nabla(\nabla\cdot u),w\right)_{\bar{0}}=-\left(\nabla\cdot u,\nabla\cdot w\right)_{0}.
      \end{equation}
      
      Plugging equations $(\ref{15})$ and $(\ref{16})$ into equation $(\ref{14})$ and rearranging terms yield
      \begin{equation*}
      \left(u^{n+1}_{t}-u^{n-1}_{t},w\right)_{\bar{0}}=-\frac{1}{\rho}\left[(\lambda_{1}+\lambda_{2})\left(\int_{t_{n-1}}^{t_{n+1}}\nabla\cdot udt,\nabla\cdot w\right)_{\bar{0}}+\lambda_{2}\left(\int_{t_{n-1}}^{t_{n+1}}\overline{\nabla}udt,\overline{\nabla}w\right)_{\bar{0}}-\right.
      \end{equation*}
      \begin{equation}\label{17}
      \left.\left(\int_{t_{n-1}}^{t_{n+1}}F(u)dt,w\right)_{\bar{0}}\right].
      \end{equation}

      To get a temporal second-order explicit numerical scheme, the term $u_{t}$ should be approximated at the points $t_{n-1}$, $t_{n}$ and $t_{n+1}$, while the integrands $\nabla\cdot u$, $\nabla u$ and $F(u)$ must be interpolated at the points $t_{n-1}$ and $t_{n}$. The approximation of the vector function $u(t)$ at the points $t_{n-1}$, $t_{n}$ and $t_{n+1}$, results in
      \begin{equation*}
        u(t)=\frac{(t-t_{n-1})(t-t_{n})}{(t_{n+1}-t_{n-1})(t_{n+1}-t_{n})}u^{n+1}+\frac{(t-t_{n-1})(t-t_{n+1})}{(t_{n}-t_{n-1})(t_{n}-t_{n+1})}u^{n}+
       \end{equation*}
       \begin{equation*}
       \frac{(t-t_{n})(t-t_{n+1})}{(t_{n-1}-t_{n})(t_{n-1}-t_{n+1})}u^{n-1}+\frac{1}{6}(t-t_{n-1})(t-t_{n})(t-t_{n+1})u_{3t}(\epsilon(t)),
       \end{equation*}
       where $\epsilon(t)$ is between the minimum and maximum of $t_{n-1}$, $t_{n}$, $t_{n+1}$ and $t$. This is equivalent to
       \begin{equation*}
        u(t)=\frac{1}{2\sigma^{2}}[(t-t_{n-1})(t-t_{n})u^{n+1}-2(t-t_{n-1})(t-t_{n+1})u^{n}+(t-t_{n})(t-t_{n+1})u^{n-1}]+
       \end{equation*}
       \begin{equation*}
       \frac{1}{6}(t-t_{n-1})(t-t_{n})(t-t_{n+1})u_{3t}(\epsilon(t)).
       \end{equation*}
       
       The time derivative provides
       \begin{equation*}
        u_{t}(t)=\frac{1}{2\sigma^{2}}[(2t-t_{n-1}-t_{n})u^{n+1}-2(2t-t_{n-1}-t_{n+1})u^{n}+(2t-t_{n}-t_{n+1})u^{n-1}]+
       \end{equation*}
       \begin{equation*}
       \frac{1}{6}\{(t-t_{n-1})(t-t_{n})(t-t_{n+1})\frac{d}{dt}(u_{3t}(\epsilon(t)))+u_{3t}(\epsilon(t))\frac{d}{dt}[(t-t_{n-1})(t-t_{n})(t-t_{n+1})]\}.
       \end{equation*}
       
        Applying $u_{t}(t)$ at the discrete points $t_{n-1}$ and $t_{n+1}$ and performing direct calculations to obtain
       \begin{equation}\label{18}
        u^{n+1}_{t}-u^{n-1}_{t}=\frac{2}{\sigma}(u^{n+1}-2u^{n}+u^{n-1})+\frac{\sigma^{2}}{3}(u_{3t}(\epsilon(t_{n+1}))-u_{3t}(\epsilon(t_{n-1}))).
      \end{equation}

      The application of the Mean Value Theorem gives
      \begin{equation}\label{19}
      u_{3t}(\epsilon(t_{n+1}))-u_{3t}(\epsilon(t_{n-1}))=(\epsilon(t_{n+1})-\epsilon(t_{n-1}))u_{4t}(\overline{\epsilon}),
      \end{equation}
      where $\overline{\epsilon}\in(\min\{\epsilon(t_{n+1}),\epsilon(t_{n-1})\},\text{\,}\max\{\epsilon(t_{n+1}),\epsilon(t_{n-1})\})$. But, $\epsilon(t_{n+1}),\epsilon(t_{n-1})\in(t_{n-1},t_{n+1})$, so $-2\sigma<\epsilon(t_{n+1})-\epsilon(t_{n-1})<2\sigma$. Utilizing this, along with approximation $(\ref{19})$, equation $(\ref{18})$ becomes
      \begin{equation}\label{20}
      u^{n+1}_{t}-u^{n-1}_{t}=\frac{2}{\sigma}(u^{n+1}-2u^{n}+u^{n-1})+\frac{\sigma^{2}}{3}(\epsilon(t_{n+1})-\epsilon(t_{n-1}))u_{4t}(\overline{\epsilon}),
      \end{equation}
      where
      \begin{equation}\label{20a}
      \|(\epsilon(t_{n+1})-\epsilon(t_{n-1}))u_{4t}(\overline{\epsilon})\|_{\bar{0}}\leq2\sigma\|u_{4t}(\overline{\epsilon})\|_{\bar{0}}.
      \end{equation}

      In addition,
      \begin{equation*}
      \nabla\cdot u(t)=\frac{t-t_{n-1}}{t_{n}-t_{n-1}} \nabla\cdot u^{n}+\frac{t-t_{n}}{t_{n-1}-t_{n}} \nabla\cdot u^{n-1}+\frac{1}{2}(t-t_{n-1})(t-t_{n}) \nabla\cdot u_{2t}(\epsilon_{1}(t))=
       \end{equation*}
       \begin{equation}\label{21}
      \frac{1}{\sigma}[(t-t_{n-1})\nabla\cdot u^{n}-(t-t_{n})\nabla\cdot u^{n-1}]+\frac{1}{2}(t-t_{n-1})(t-t_{n}) \nabla\cdot u_{2t}(\epsilon_{1}(t)),
      \end{equation}
      where $\epsilon(t)$ is between the minimum and maximum of $t_{n-1}$, $t_{n}$ and $t$. Since the function $t\mapsto\nabla\cdot u_{2t}(\epsilon_{1}(\cdot))$ is continuous on $[0,\text{\,}T]$ and the function $t\mapsto\nabla\cdot (t-t_{n-1})(t-t_{n})$ does not change sign on the intervals $[t_{n-1},\text{\,}t_{n}]$ and $[t_{n},\text{\,}t_{n+1}]$, integrating both sides of approximations $(\ref{21})$ and utilizing the Weight Mean Value theorem, this yields
      \begin{equation*}
      \int_{t_{n-1}}^{t_{n+1}}\nabla\cdot u(t)dt=\frac{1}{\sigma}\left[\nabla\cdot u^{n}\int_{t_{n-1}}^{t_{n+1}}(t-t_{n-1})dt-\nabla\cdot u^{n-1}\int_{t_{n-1}}^{t_{n+1}}(t-t_{n})dt\right]+
      \end{equation*}
      \begin{equation*}
      \frac{1}{2}\int_{t_{n-1}}^{t_{n+1}}(t-t_{n-1})(t-t_{n}) \nabla\cdot u_{2t}(\epsilon_{1}(t))dt=\frac{1}{2\sigma}\left(\nabla\cdot u^{n}[(t-t_{n-1})^{2}]_{t_{n-1}}^{t_{n+1}}
      -\nabla\cdot u^{n-1}[(t-t_{n})^{2}]_{t_{n-1}}^{t_{n+1}}\right)+
      \end{equation*}
      \begin{equation*}
      \frac{1}{2}\left(\nabla\cdot u_{2t}(\overline{\epsilon}_{1})\int_{t_{n-1}}^{t_{n}}(t-t_{n-1})(t-t_{n})dt+\nabla\cdot u_{2t}(\overline{\overline{\epsilon}}_{1})
      \int_{t_{n}}^{t_{n+1}}(t-t_{n-1})(t-t_{n})dt\right)=
      \end{equation*}
      \begin{equation}\label{22}
      2\sigma\nabla\cdot u^{n}+\frac{\sigma^{3}}{12}(5\nabla\cdot u_{2t}(\overline{\overline{\epsilon}}_{1})-\nabla\cdot u_{2t}(\overline{\epsilon}_{1})),
      \end{equation}
      where $\overline{\epsilon}_{1}\in(t_{n-1},t_{n})$ and $\overline{\overline{\epsilon}}_{1}\in(t_{n},t_{n+1})$.\\

      In a similar manner, one easily shows that
      \begin{equation}\label{23}
      \int_{t_{n-1}}^{t_{n+1}}\overline{\nabla}u(t)dt=2\sigma\overline{\nabla}u^{n}+\frac{\sigma^{3}}{12}(5\overline{\nabla}u_{2t}(\overline{\overline{\epsilon}}_{2})-\overline{\nabla} u_{2t}(\overline{\epsilon}_{2})),
      \end{equation}
      where $\overline{\epsilon}_{2}\in(t_{n-1},t_{n})$ and $\overline{\overline{\epsilon}}_{2}\in(t_{n},t_{n+1})$.
      \begin{equation}\label{24}
      \int_{t_{n-1}}^{t_{n+1}}F(u)(t)dt=2\sigma F(u^{n})+\frac{\sigma^{3}}{12}[5(F(u))_{2t}(\overline{\overline{\epsilon}}_{3})-(F(u))_{2t}(\overline{\epsilon}_{3})],
      \end{equation}
      where $\overline{\epsilon}_{3}\in(t_{n-1},t_{n})$ and $\overline{\overline{\epsilon}}_{3}\in(t_{n},t_{n+1})$. It's worth noticing to recall that $F(u)\in[W_{2}^{3}(0,T;\text{\,}H^{3})]^{3}$. Substituting approximations $(\ref{20})$, $(\ref{22})$, $(\ref{23})$ and $(\ref{24})$ into equation $(\ref{17})$, and rearranging terms, this results in
      \begin{equation*}
      \left(u^{n+1}-2u^{n}+u^{n-1},w\right)_{\bar{0}}=-\frac{\sigma^{2}}{\rho}\left[(\lambda_{1}+\lambda_{2})\left(\nabla\cdot u^{n},\nabla\cdot w\right)_{0}+
      \lambda_{2}\left(\overline{\nabla}u^{n},\overline{\nabla}w\right)_{*}-\left(F(u^{n}),w\right)_{\bar{0}}\right]-
      \end{equation*}
      \begin{equation*}
      \frac{\sigma^{3}}{3}\left((\epsilon(t_{n+1})-\epsilon(t_{n-1}))u_{4t}(\overline{\epsilon}),w\right)_{\bar{0}}+\frac{(\lambda_{1}+\lambda_{2})\sigma^{4}}{24\rho}\left(\nabla\cdot (u_{2t}(\overline{\epsilon}_{1})-5u_{2t}(\overline{\overline{\epsilon}}_{1})),\nabla\cdot w\right)_{0}+
      \end{equation*}
      \begin{equation}\label{25}
      \frac{\lambda_{2}\sigma^{4}}{24\rho}\left(\overline{\nabla}(u_{2t}(\overline{\epsilon}_{2})-5u_{2t}(\overline{\overline{\epsilon}}_{2})),\overline{\nabla}w\right)_{*}+
      \frac{\sigma^{4}}{24\rho}\left(5(F(u))_{2t}(\overline{\overline{\epsilon}}_{3})-(F(u))_{2t}(\overline{\epsilon}_{3}),w\right)_{\bar{0}}.
      \end{equation}
      
      Since $\overline{\epsilon}_{1},\overline{\epsilon}_{2},\overline{\epsilon}_{3}\in(t_{n-1},t_{n})$ and $\overline{\overline{\epsilon}}_{1},\overline{\overline{\epsilon}}_{2},\overline{\overline{\epsilon}}_{3}\in(t_{n},t_{n+1})$, without loss of this generality and by the sake of readability, we should assume that $\overline{\epsilon}_{1}=\overline{\epsilon}_{2}=\overline{\epsilon}_{3}=\theta_{1}$ and $\overline{\overline{\epsilon}}_{1}=\overline{\overline{\epsilon}}_{2}= \overline{\overline{\epsilon}}_{3}=\theta_{2}$. However, this assumption will not compromise neither the result on stability nor the one on error estimates. Now, utilizing the bilinear operator $B$ given in equation $(\ref{11})$, it is easy to see that
      \begin{equation*}
      \left(u^{n+1},w\right)_{\bar{0}}=\left(2u^{n}-u^{n-1},w\right)_{\bar{0}}-\frac{\sigma^{2}}{\rho}[B(u^{n},w)-\left(F(u^{n}),w\right)_{\bar{0}}]+
      \frac{\sigma^{4}}{24\rho}B(u_{2t}(\theta_{1})-5u_{2t}(\theta_{2}),w)+
      \end{equation*}
      \begin{equation}\label{26}
      \frac{\sigma^{4}}{24\rho}\left(5(F(u))_{2t}(\theta_{2})-(F(u))_{2t}(\theta_{1}),w\right)_{\bar{0}}-\frac{\sigma^{3}}{3}\left((\epsilon(t_{n+1})-\epsilon(t_{n-1}))
      u_{4t}(\overline{\epsilon}),w\right)_{\bar{0}}.
      \end{equation}

      Neglecting the error term: $\frac{\sigma^{4}}{24\rho}B(u_{2t}(\theta_{1})-5u_{2t}(\theta_{2}),w)+\frac{\sigma^{4}}{24\rho}\left(5(F(u))_{2t}(\theta_{2})-
      (F(u))_{2t}(\theta_{1}),w\right)_{\bar{0}}-\\      \frac{\sigma^{3}}{3}\left((\epsilon(t_{n+1})-\epsilon(t_{n-1}))u_{4t}(\overline{\epsilon}),w\right)_{\bar{0}}$, replacing the analytical solution $u$ with the computed one $u_{h}\in\mathcal{U}_{h}$, equation $(\ref{26})$ can be approximated as
      \begin{equation}\label{27}
      \left(u_{h}^{n+1},w\right)_{\bar{0}}=\left(2u_{h}^{n}-u_{h}^{n-1},w\right)_{\bar{0}}-\frac{\sigma^{2}}{\rho}[B(u_{h}^{n},w)-\left(F(u_{h}^{n}),w\right)_{\bar{0}}].
      \end{equation}

      In addition, multiplying the second equation in system $(\ref{1})$ by $\tau\in\mathcal{M}$ and utilizing the inner product $\left(\cdot,\cdot\right)_{*}$, given in relation $(\ref{8})$ yield
      \begin{equation}\label{28}
      \left(\psi^{n},\tau\right)_{*}=\lambda_{1}\left((\nabla\cdot u_{h}^{n})\mathcal{I},\tau\right)_{*}+\lambda_{2}\left(\nabla u_{h}^{n}+(\nabla u_{h}^{n})^{T},\tau\right)_{*}.
      \end{equation}

      A combination of approximations $(\ref{27})$-$(\ref{28})$ provides the desired algorithm, that is, given $(u_{h}^{n},\psi_{h}^{n}),$ $(u_{h}^{n-1},\psi_{h}^{n-1})\in\mathcal{P}_{h}=\mathcal{U}_{h}\times\mathcal{M}_{h}$, find $(u_{h}^{n+1},\psi_{h}^{n+1})\in\mathcal{P}_{h}$, for $n=1,2,...,N-1$, so that
      \begin{equation}\label{s1}
      \left(u_{h}^{n+1},w\right)_{\bar{0}}=\left(2u_{h}^{n}-u_{h}^{n-1},w\right)_{\bar{0}}-\frac{\sigma^{2}}{\rho}[B(u_{h}^{n},w)-\left(F(u_{h}^{n}),w\right)_{\bar{0}}],\text{\,\,\,\,}\forall w\in \mathcal{U},
      \end{equation}
      \begin{equation}\label{s2}
      \left(\psi_{h}^{n+1},\tau\right)_{*}=\lambda_{1}\left((\nabla\cdot u_{h}^{n+1})\mathcal{I},\tau\right)_{*}+\lambda_{2}\left(\nabla u_{h}^{n+1}+(\nabla u_{h}^{n+1})^{T},\tau\right)_{*}, \text{\,\,\,\,}\forall \tau\in \mathcal{M},
      \end{equation}
      where $\mathcal{W}$ and $\mathcal{M}$ are defined in equation $(\ref{5})$, with initial conditions
      \begin{equation}\label{s3}
      u_{h}^{0}=u_{0},\text{\,\,}u_{h}^{1}=\widetilde{v},\text{\,\,}\psi_{h}^{0}=\lambda_{1}(\nabla\cdot u_{0})\mathcal{I}+\lambda_{2}(\nabla u_{0}+(\nabla u_{0})^{T}),
      \text{\,\,}\psi_{h}^{1}=\lambda_{1}(\nabla\cdot u_{h}^{1})\mathcal{I}+\lambda_{2}(\nabla u_{h}^{1}+(\nabla u_{h}^{1})^{T}), \text{\,\,on\,\,}\overline{\Omega}=\Omega\cup\Gamma,
      \end{equation}
      and boundary condition
      \begin{equation}\label{s4}
      u_{h}^{n}=0,\text{\,\,\,\,},\psi_{h}^{n}=0,\text{\,\,\,\,for\,\,\,\,}n=0,1,...,N,\text{\,\,\,\,}\text{\,\,\,\,on\,\,\,\,\,}\Gamma,
      \end{equation}
      where
      \begin{equation}\label{s5}
      \widetilde{v}=u_{0}+\sigma u_{t}^{0}=u_{0}+\sigma v,
      \end{equation}
      is the second order approximation of the term $u^{1}=u(t_{1})$ obtained by the use of the Taylor series expansion.\\

       It's worth mentioning that the second equation in $(\ref{1})$ indicates that the stress tensor $\psi$ does not deal with a differential equation and should be obtained by simple calculations once the displacement $u_{h}$ is known. For this reason, the analysis of stability and error estimates of the developed computational technique $(\ref{s1})$-$(\ref{s5})$ will be restricted to equations satisfied by the displacement $u_{h}$. Additionally, both stability and error estimates of the proposed explicit approach will be analyzed under the following assumptions.
       \begin{description}
       \item[(i)] The generalized sequence $\{\mathcal{U}_{h}\}_{h>0}$, of finite element subspaces approximating $\mathcal{U}$ with order $O(h)$ are utilized in the fluid region. As discussed in \cite{1enarxiv}, the corresponding inverse inequality is given by
         \begin{equation}\label{31}
       \left\|\frac{\partial w}{\partial x_{l}}\right\|_{\bar{0}}\leq C_{p}h^{-1}\|w\|_{\bar{0}},\text{\,\,\,\,for\,\,\,\,}l=1,2,3,\text{\,\,\,}w\in\mathcal{U},
      \end{equation}
      where $C_{p}>0$ is a constant which does not depend on the time step $\sigma$ and space step $h$.
         \item[(ii)] The time step $\sigma$ and the mesh grid $h$ satisfy the following requirement
         \begin{equation}\label{32}
          \sigma\leq C_{ts}h,\text{\,\,\,\,}\text{\,\,\,\,where\,\,\,\,\,}0<C_{ts}<\sqrt{2\rho C_{1}},
      \end{equation}
      where $C_{1}=\frac{1}{18}C_{p}^{-2}(\lambda_{1}+\lambda_{2})^{-1}$. Since the right side of estimate $(\ref{32})$ depends on the mesh size $"h"$, the developed numerical scheme is stable under an appropriate time step limitation. Additionally, when $\rho\geq1$, this estimate shows that the new algorithm advances the computed solution with a maximum allowable time step. Finally, the time step limitation is more attractive and it is well known in the literature as the Courant-Friedrichs-Lewy (CFL) condition for stability of explicit numerical methods applied to hyperbolic PDEs.
         \item[(iii)] The nonlinear source term $F(u)$ falls in the space $[W^{3}_{2}(0,T;\text{\,}H^{3})]^{3}$. Using the first equation in system $(\ref{1})$, this suggests that the exact solution $u=(u_{1},u_{2},u_{3})^{T}\in[W^{5}_{2}(0,T;\text{\,}H^{5})]^{3}$. Moreover, there exists a positive constant $\widehat{C}$, so that
       \begin{equation}\label{33}
       \||u|\|_{\overline{5},5}=\sqrt{\underset{l=1}{\overset{3}\sum}\||u_{l}|\|_{5,5}^{2}}\leq\widehat{C}.
      \end{equation}

      In addition, we assume that the solution of the second equation in system $(\ref{1})$, $\psi\in L^{2}(0,T;\text{\,}\mathcal{M}_{3}(H^{5}))$ and the finite element spaces $\mathcal{U}_{h}$ and $\mathcal{M}_{h}$ defined in equations $(\ref{9})$-$(\ref{10})$ satisfy the approximation properties of piecewise polynomials of degrees $4$ and $5$. That is,
      \begin{equation*}
       \underset{u_{h}\in\mathcal{U}_{h}}{\inf}\|u_{h}-u\|_{\bar{0}}\leq \gamma_{0}h^{5}\|u\|_{[H^{5}]^{3}},
      \end{equation*}
      \begin{equation*}
       \underset{u_{h}\in\mathcal{U}_{h}}{\inf}\|u_{h}-u\|_{B}\leq \gamma_{1}h^{4}\|u\|_{[H^{4}]^{3}},
       \end{equation*}
       \begin{equation}\label{33a}
       \underset{\psi_{h}\in\mathcal{M}_{h}}{\inf}\|\psi_{h}-\psi\|_{*}\leq \gamma_{2}h^{5}\|\psi\|_{\mathcal{M}_{3}(H^{5})},
      \end{equation}
      where $\gamma_{l}$, for $l=0,1,2,$ are positive constants which do not depend on the mesh grid $h$ and time step $\sigma$.
       \end{description}
       
      The following Poincar\'{e}-Friedrichs inequality will play a crucial role in the analysis of stability and error estimates.
      \begin{equation}\label{34}
      \|v\|_{0}\leq C_{\Omega}\|\nabla v\|_{\bar{0}},\text{\,\,\,\,\,\,}\forall v\in W_{2}^{1}(\Omega),
      \end{equation}
      where $C_{\Omega}$, is a positive constant independent of the mesh grid $h$ and time step $\sigma$.

     \section{Stability analysis and error estimates of the new algorithm}\label{sec3}
     
     \text{\,\,\,\,\,\,\,\,\,\,}This Section considers the stability and error estimates of the developed explicit computational approach $(\ref{s1})$-$(\ref{s5})$ applied to the three-dimensional nonlinear elastodynamic sine-Gordon problem $(\ref{1})$-$(\ref{3})$. Since the stress tensor $"\psi"$ is not associated with a differential equation and should be directly computed as a function of the displacement $"u"$, the analysis on stability and error estimates will be restricted on the equations satisfied by the displacement $u_{h}$.

     \begin{lemma}\label{l1}
     For any $v,w\in\mathcal{U}$, the bilinear form $B(\cdot,\cdot)$ defined by equation $(\ref{11})$ is symmetric and satisfies the following estimates
     \begin{equation}\label{37}
     B(v,w)\leq (\lambda_{1}+4\lambda_{2})\|v\|_{\bar{1}}\|w\|_{\bar{1}}\text{\,\,\,\,\,and\,\,\,\,\,}B(v,v)\geq\frac{\lambda_{1}+\lambda_{2}}{2C_{\Omega}}
     \min\{1,C_{\Omega}\}\|v\|_{\bar{1}}^{2},
     \end{equation}
     where $\lambda_{1}$ and $\lambda_{2}$ are the nonnegative parameters given in equation $(\ref{1})$ and $C_{\Omega}$ is the positive constant given in estimate $(\ref{34})$.
     \end{lemma}

   \begin{proof}
    Utilizing equation $(\ref{11})$, it holds
    \begin{equation*}
    B(v,w)=(\lambda_{1}+\lambda_{2})\left(\overline{\nabla}v,\overline{\nabla}w\right)_{*}+\lambda_{2}\left(\nabla\cdot v,\nabla\cdot w\right)_{0} =(\lambda_{1}+\lambda_{2})\left(\overline{\nabla}w,\overline{\nabla}v\right)_{*}+\lambda_{2}\left(\nabla\cdot w,\nabla\cdot v\right)_{0}=B(w,v).
   \end{equation*}
   
   Thus, the bilinear operator $B(\cdot,\cdot)$ is symmetric. Regarding the proof of the estimates provided in relation $(\ref{37})$, we refer the readers to \cite{1enarxiv}.
   This ends the proof of Lemma $\ref{l1}$.
   \end{proof}
   It's worth mentioning that Lemma $\ref{l1}$ shows that the bilinear operator $B(\cdot,\cdot)$ defines a scalar product on the Sobolev space $\mathcal{U}\times\mathcal{U}$. Let $\|\cdot\|_{B}$ be the norm associated with the scalar product $B(\cdot,\cdot)$. Hence, the norm $\|\cdot\|_{B}$ is defined as: for $v\in\mathcal{U}$,
   \begin{equation}\label{38}
    \|v\|_{B}^{2}=B(v,v)=(\lambda_{1}+\lambda_{2})\|\overline{\nabla}v\|_{*}^{2}+\lambda_{2}\|\nabla\cdot v\|_{0}^{2}.
   \end{equation}

   We introduce the norm, $\|\cdot\|_{\bar{0},B}$, defined on $\mathcal{U}$ as:
   \begin{equation}\label{38a}
    \|v\|_{\bar{0},B}=\sqrt{\|v\|_{\bar{0}}^{2}+\|v\|_{B}^{2}},\text{\,\,\,\,\,}\forall v\in\mathcal{U}.
   \end{equation}

     \begin{theorem} \label{t1} (Stability analysis and error estimates).
      Consider $u\in [W^{5}_{2}(0,T;\text{\,}H^{5})]^{3}$, be the exact solution of the initial-boundary value problem $(\ref{1})$-$(\ref{3})$ and let $u_{h}(t)\in\mathcal{U}_{h}$, for $t\in[0,\text{\,}T]$, be the approximate solution provided by the proposed approach $(\ref{s1})$-$(\ref{s5})$. Under the time step requirement $(\ref{32})$, the following estimates hold
     \begin{equation*}
     \|u_{h}^{n+1}\|_{\bar{0},B}\leq \|u^{n+1}\|_{\bar{0},B}+\frac{1}{\sqrt{\beta_{2}}}\exp\left(\frac{8T}{\beta_{0}}+\frac{C_{1}C_{3}T}{4\rho}e^{16T\beta_{0}^{-1}}\right)
     \left[4C_{ts}^{2}\gamma_{0}^{2}\rho\|u^{1}\|_{[H^{5}]^{3}}^{2}+\right.\frac{1}{576\rho}\left(26\|(f(u))_{2t}\|_{\bar{0},B}^{2}\right.
     \end{equation*}
     \begin{equation}\label{35}
      \left.\left.+128\rho^{2}\||u_{4t}|\|_{\bar{0},\infty}^{2}+52\left((\lambda_{1}+\lambda_{2})^{2}\||\nabla\cdot\overline{\nabla}u_{2t}|\|_{\bar{0},\infty}^{2}+
     \lambda_{2}^{2}\||\nabla(\nabla\cdot u_{2t})|\|_{\bar{0},\infty}^{2}\right)\right)\right]^{\frac{1}{2}}(\sigma^{2}+h^{3}),
     \end{equation}
     \begin{equation*}
     \|u_{h}^{n+1}-u^{n+1}\|_{\bar{0},B}\leq \frac{1}{\sqrt{\beta_{2}}}\exp\left(\frac{8T}{\beta_{0}}+\frac{C_{1}C_{3}T}{4\rho}e^{16T\beta_{0}^{-1}}\right)
     \left[4C_{ts}^{2}\gamma_{0}^{2}\rho\|u^{1}\|_{[H^{5}]^{3}}^{2}+\right.\frac{1}{576\rho}\left(26\|(f(u))_{2t}\|_{\bar{0},B}^{2}\right.
     \end{equation*}
     \begin{equation}\label{36}
      \left.\left.+128\rho^{2}\||u_{4t}|\|_{\bar{0},\infty}^{2}+52\left((\lambda_{1}+\lambda_{2})^{2}\||\nabla\cdot\overline{\nabla}u_{2t}|\|_{\bar{0},\infty}^{2}+
     \lambda_{2}^{2}\||\nabla(\nabla\cdot u_{2t})|\|_{\bar{0},\infty}^{2}\right)\right)\right]^{\frac{1}{2}}(\sigma^{2}+h^{3}),
     \end{equation}
      for $n=0,1,...,N-1$, where $C_{1}=\frac{C_{p}^{-2}}{18(\lambda_{1}+\lambda_{2})}$,$C_{2}=\frac{C_{\Omega}}{\lambda_{1}+\lambda_{2}}
      \frac{\max\{1,C_{\Omega}\}}{\min\{1,C_{\Omega}\}}$, $\beta_{2}=\min\{1,\text{\,}\frac{1}{2C_{2}}\}$, $C_{3}>0$ denotes the Lipschitz constant associated with the function $F(u)$, $\beta_{0}=1-\frac{C_{ts}^{2}}{2\rho C_{1}}>0$, $\gamma_{0}$ is the constant given in relation $(\ref{33a})$ while $C_{ts}$ is the constant defined in estimates $(\ref{32})$ and $C_{p}$ is the positive constant defined in estimate $(\ref{31})$.
    \end{theorem}

    The following Lemma plays a crucial role in the proof of Theorem $\ref{t1}$.

   \begin{lemma}\label{l2}\cite{1enarxiv}
     For every $v\in\mathcal{U}_{h}$, the following inequalities are satisfied
     \begin{equation}\label{39}
     \sqrt{C_{1}}h\|v\|_{B}\leq \|v\|_{\bar{0}}\leq \sqrt{C_{2}}\|v\|_{B},
     \end{equation}
     where $C_{1}$ and $C_{2}$ are given in Theorem $\ref{t1}$.
     \end{lemma}

    \begin{proof}(of Theorem $\ref{t1}$).
     Set $e_{h}^{l}=u_{h}^{l}-u^{l}$, combine equations $(\ref{26})$ and $(\ref{s1})$, and rearranging terms to get
     \begin{equation*}
      \left(e_{h}^{n+1}-2e_{h}^{n}+e_{h}^{n-1},w\right)_{\bar{0}}+\frac{\sigma^{2}}{\rho}B(e_{h}^{n},w)=\frac{\sigma^{2}}{\rho}\left(F(u_{h}^{n})-F(u^{n}),w\right)_{\bar{0}}-
      \frac{\sigma^{4}}{24\rho}B(u_{2t}(\theta_{1})-5u_{2t}(\theta_{2}),w)-
      \end{equation*}
       \begin{equation*}
      \frac{\sigma^{4}}{24\rho}\left(5(F(u))_{2t}(\theta_{2})-(F(u))_{2t}(\theta_{1}),w\right)_{\bar{0}}+\frac{\sigma^{3}}{3}\left((\epsilon(t_{n+1})-
      \epsilon(t_{n-1}))u_{4t}(\overline{\epsilon}),w\right)_{\bar{0}},\text{\,\,\,\,}\forall w\in \mathcal{U}.
      \end{equation*}

      Taking $w=e_{h}^{n+1}-e_{h}^{n-1}$, this becomes
      \begin{equation*}
      \left(e_{h}^{n+1}-2e_{h}^{n}+e_{h}^{n-1},e_{h}^{n+1}-e_{h}^{n-1}\right)_{\bar{0}}+\frac{\sigma^{2}}{\rho}B(e_{h}^{n},e_{h}^{n+1}-e_{h}^{n-1})=\frac{\sigma^{2}}{\rho}
      \left(F(u_{h}^{n})-F(u^{n}),e_{h}^{n+1}-e_{h}^{n-1}\right)_{\bar{0}}-
      \end{equation*}
      \begin{equation*}
      \frac{\sigma^{4}}{24\rho}B(u_{2t}(\theta_{1})-5u_{2t}(\theta_{2}),e_{h}^{n+1}-e_{h}^{n-1})-\frac{\sigma^{4}}{24\rho}\left(5(F(u))_{2t}(\theta_{2})-(F(u))_{2t}(\theta_{1}),
      e_{h}^{n+1}-e_{h}^{n-1}\right)_{\bar{0}}+
      \end{equation*}
       \begin{equation}\label{40}
      \frac{\sigma^{3}}{3}\left((\epsilon(t_{n+1})-\epsilon(t_{n-1}))u_{4t}(\overline{\epsilon}),e_{h}^{n+1}-e_{h}^{n-1}\right)_{\bar{0}}.
      \end{equation}

      Simple calculations give
      \begin{equation*}
      \left(e_{h}^{n+1}-2e_{h}^{n}+e_{h}^{n-1},e_{h}^{n+1}-e_{h}^{n-1}\right)_{\bar{0}}=\left(e_{h}^{n+1}-e_{h}^{n},e_{h}^{n+1}-e_{h}^{n-1}\right)_{\bar{0}}-
      \left(e_{h}^{n}-e_{h}^{n-1},e_{h}^{n+1}-e_{h}^{n-1}\right)_{\bar{0}}=
     \end{equation*}
     \begin{equation*}
      \left(e_{h}^{n+1}-e_{h}^{n},e_{h}^{n+1}-e_{h}^{n}\right)_{\bar{0}}+\left(e_{h}^{n+1}-e_{h}^{n},e_{h}^{n}-e_{h}^{n-1}\right)_{\bar{0}}-
      \left(e_{h}^{n}-e_{h}^{n-1},e_{h}^{n+1}-e_{h}^{n}\right)_{\bar{0}}-
     \end{equation*}
     \begin{equation}\label{41}
      \left(e_{h}^{n}-e_{h}^{n-1},e_{h}^{n}-e_{h}^{n-1}\right)_{\bar{0}}=\|e_{h}^{n+1}-e_{h}^{n}\|_{\bar{0}}^{2}-\|e_{h}^{n}-e_{h}^{n-1}\|_{\bar{0}}^{2}.
     \end{equation}
     
     In a similar manner, it holds
     \begin{equation*}
      B(e_{h}^{n},e_{h}^{n+1}-e_{h}^{n-1})=-\frac{1}{2}B(-2e_{h}^{n},e_{h}^{n+1}-e_{h}^{n-1})=-\frac{1}{2}[B(e_{h}^{n+1}-2e_{h}^{n}+e_{h}^{n-1},e_{h}^{n+1}-e_{h}^{n-1})-
     \end{equation*}
     \begin{equation}\label{42}
      B(e_{h}^{n+1}+e_{h}^{n-1},e_{h}^{n+1}-e_{h}^{n-1})]=-\frac{1}{2}[\|e_{h}^{n+1}-e_{h}^{n}\|_{B}^{2}-\|e_{h}^{n}-e_{h}^{n-1}\|_{B}^{2}-\|e_{h}^{n+1}\|_{B}^{2}+
      \|e_{h}^{n-1}\|_{B}^{2}].
     \end{equation}

      Substituting estimates $(\ref{41})$ and $(\ref{42})$ into approximation $(\ref{40})$ and rearranging terms, this yields
      \begin{equation*}
      (\|e_{h}^{n+1}-e_{h}^{n}\|_{\bar{0}}^{2}-\|e_{h}^{n}-e_{h}^{n-1}\|_{\bar{0}}^{2})+\frac{\sigma^{2}}{2\rho}\left[-(\|e_{h}^{n+1}-e_{h}^{n}\|_{B}^{2}-
      \|e_{h}^{n}-e_{h}^{n-1}\|_{B}^{2})+(\|e_{h}^{n+1}\|_{B}^{2}-\|e_{h}^{n}\|_{B}^{2})+\right.
      \end{equation*}
      \begin{equation*}
      \left.(\|e_{h}^{n}\|_{B}^{2}-\|e_{h}^{n-1}\|_{B}^{2})\right]=\frac{\sigma^{2}}{\rho}\left(F(u_{h}^{n})-F(u^{n}),e_{h}^{n+1}-e_{h}^{n-1}\right)_{\bar{0}}-
      \frac{\sigma^{4}}{24\rho}\left[B(u_{2t}(\theta_{1})-5u_{2t}(\theta_{2}),e_{h}^{n+1}-e_{h}^{n-1})+\right.
      \end{equation*}
       \begin{equation}\label{43}
       \left.\left(5(F(u))_{2t}(\theta_{2})-(F(u))_{2t}(\theta_{1}),e_{h}^{n+1}-e_{h}^{n-1}\right)_{\bar{0}}\right]+
      \frac{\sigma^{3}}{3}\left((\epsilon(t_{n+1})-\epsilon(t_{n-1}))u_{4t}(\overline{\epsilon}),e_{h}^{n+1}-e_{h}^{n-1}\right)_{\bar{0}}.
      \end{equation}

      The application of the Cauchy-Schwarz inequality provides
     \begin{equation}\label{44}
     \frac{\sigma^{2}}{\rho}\left(F(u_{h}^{n})-F(u^{n}),e_{h}^{n+1}-e_{h}^{n-1}\right)_{\bar{0}}\leq \frac{\sigma^{3}}{4\rho^{2}}\|F(u_{h}^{n})-F(u^{n})\|_{\bar{0}}^{2}+
     \sigma\|e_{h}^{n+1}-e_{h}^{n-1}\|_{\bar{0}}^{2}.
     \end{equation}

     Since the function $F(v)$ satisfies a Lipschitz condition in variable $v$, there is a positive constant $C_{3}$ independent of the time step $\sigma$ and mesh size $h$, so that
      $\|F(u_{h}^{n})-F(u^{n})\|_{\bar{0}}\leq \sqrt{C_{3}}\|u_{h}^{n}-u^{n}\|_{\bar{0}}=\sqrt{C_{3}}\|e_{h}^{n}\|_{\bar{0}}$. Utilizing estimate $(\ref{39})$, this implies $\|F(u_{h}^{n})-F(u^{n})\|_{\bar{0}}^{2}\leq C_{3}C_{2}\|e_{h}^{n}\|_{B}^{2}$. Substituting this into estimate $(\ref{44})$ to obtain
      \begin{equation}\label{45}
     \frac{\sigma^{2}}{\rho}\left(F(u_{h}^{n})-F(u^{n}),e_{h}^{n+1}-e_{h}^{n-1}\right)_{\bar{0}}\leq \frac{C_{3}C_{2}\sigma^{3}}{4\rho^{2}}\|e_{h}^{n}\|_{B}^{2}+
     \sigma\|e_{h}^{n+1}-e_{h}^{n-1}\|_{\bar{0}}^{2}.
     \end{equation}
     
     Additionally,
     \begin{equation*}
      -\frac{\sigma^{4}}{24\rho}\left(5(F(u))_{2t}(\theta_{2})-(F(u))_{2t}(\theta_{1}),e_{h}^{n+1}-e_{h}^{n-1}\right)_{\bar{0}}\leq \frac{\sigma^{7}}{(48\rho)^{2}}\|5(F(u))_{2t}(\theta_{2})-(F(u))_{2t}(\theta_{1})\|_{\bar{0}}^{2}+
      \end{equation*}
     \begin{equation}\label{46}
     \sigma\|e_{h}^{n+1}-e_{h}^{n-1}\|_{\bar{0}}^{2}\leq \frac{2\sigma^{7}}{(48\rho)^{2}}(25\|(F(u))_{2t}(\theta_{2})\|_{\bar{0}}^{2}+
     \|(F(u))_{2t}(\theta_{1})\|_{\bar{0}}^{2})+\sigma\|e_{h}^{n+1}-e_{h}^{n-1}\|_{\bar{0}}^{2}.
     \end{equation}
     
     But it is shown below equation $(\ref{19})$ that $-2\sigma<\epsilon(t_{n+1})-\epsilon(t_{n-1})<2\sigma$. Utilizing this, direct calculations result in
     \begin{equation}\label{47}
     \frac{\sigma^{3}}{3}\left((\epsilon(t_{n+1})-\epsilon(t_{n-1}))u_{4t}(\overline{\epsilon}),e_{h}^{n+1}-e_{h}^{n-1}\right)_{\bar{0}}\leq  \frac{\sigma^{7}}{9}\|u_{4t}(\overline{\epsilon})\|_{\bar{0}}^{2}+\sigma\|e_{h}^{n+1}-e_{h}^{n-1}\|_{\bar{0}}^{2}.
     \end{equation}

     Since $u\in [W^{5}_{2}(0,T;\text{\,}H^{5})]^{3}$, so $\nabla\cdot\overline{\nabla}u_{2t},\nabla(\nabla\cdot u_{2t})\in [W^{3}_{2}(0,T;\text{\,}H^{3})]^{3}$. Using the definition of the bilinear operator $B(\cdot,\cdot)$ given by equation $(\ref{11})$ together with equations $(\ref{15})$ and $(\ref{16})$, it is easy to observe that
     \begin{equation*}
     B(u_{2t}(\theta_{1})-5u_{2t}(\theta_{2}),e_{h}^{n+1}-e_{h}^{n-1})=(\lambda_{1}+\lambda_{2})\left(\overline{\nabla}(u_{2t}(\theta_{1})-5u_{2t}(\theta_{2})),
     \overline{\nabla}(e_{h}^{n+1}-e_{h}^{n-1})\right)_{*}+
     \end{equation*}
     \begin{equation*}
     \lambda_{2}\left(\nabla\cdot(u_{2t}(\theta_{1})-5u_{2t}(\theta_{2})),\nabla\cdot(e_{h}^{n+1}-e_{h}^{n-1})\right)_{0}=
     -(\lambda_{1}+\lambda_{2})\left(\nabla\cdot\overline{\nabla}(u_{2t}(\theta_{1})-5u_{2t}(\theta_{2})),e_{h}^{n+1}-e_{h}^{n-1}\right)_{\bar{0}}-
     \end{equation*}
     \begin{equation*}
     \lambda_{2}\left(\nabla(\nabla\cdot(u_{2t}(\theta_{1})-5u_{2t}(\theta_{2}))),e_{h}^{n+1}-e_{h}^{n-1}\right)_{\bar{0}}.
     \end{equation*}

     Thus, straightforward computations yield
     \begin{equation*}
     B(u_{2t}(\theta_{1})-5u_{2t}(\theta_{2}),e_{h}^{n+1}-e_{h}^{n-1})\leq \frac{4\sigma^{7}}{(48\rho)^{2}}\left[(\lambda_{1}+\lambda_{2})^{2}
     \left(\|\nabla\cdot\overline{\nabla}u_{2t}(\theta_{1})\|_{\bar{0}}^{2}+25\|\nabla\cdot\overline{\nabla}u_{2t}(\theta_{2})\|_{\bar{0}}^{2}\right)+\right.
     \end{equation*}
     \begin{equation}\label{48}
     \left.\lambda_{2}^{2}\left(\|\nabla(\nabla\cdot u_{2t}(\theta_{1}))\|_{\bar{0}}^{2}+25\|\nabla(\nabla\cdot u_{2t}(\theta_{2}))\|_{\bar{0}}^{2}\right)\right]+
     \sigma\|e_{h}^{n+1}-e_{h}^{n-1}\|_{\bar{0}}^{2}.
     \end{equation}

     Substituting estimates $(\ref{45})$-$(\ref{48})$ into equation $(\ref{43})$, we obtain
      \begin{equation*}
      (\|e_{h}^{n+1}-e_{h}^{n}\|_{\bar{0}}^{2}-\|e_{h}^{n}-e_{h}^{n-1}\|_{\bar{0}}^{2})+\frac{\sigma^{2}}{2\rho}\left[-(\|e_{h}^{n+1}-e_{h}^{n}\|_{B}^{2}-
      \|e_{h}^{n}-e_{h}^{n-1}\|_{B}^{2})+(\|e_{h}^{n+1}\|_{B}^{2}-\|e_{h}^{n}\|_{B}^{2})+\right.
      \end{equation*}
      \begin{equation*}
      \left.(\|e_{h}^{n}\|_{B}^{2}-\|e_{h}^{n-1}\|_{B}^{2})\right]\leq \frac{C_{3}C_{2}\sigma^{3}}{4\rho^{2}}\|e_{h}^{n}\|_{B}^{2}+\frac{2\sigma^{7}}{(48\rho)^{2}}\left\{
      25\|(F(u))_{2t}(\theta_{2})\|_{\bar{0}}^{2}+\|(F(u))_{2t}(\theta_{1})\|_{\bar{0}}^{2}+128\rho^{2}\|u_{4t}(\overline{\epsilon})\|_{\bar{0}}^{2}\right.
      \end{equation*}
      \begin{equation*}
      \left.+2\left[(\lambda_{1}+\lambda_{2})^{2}\left(\|\nabla\cdot\overline{\nabla}u_{2t}(\theta_{1})\|_{\bar{0}}^{2}+25\|\nabla\cdot
       \overline{\nabla}u_{2t}(\theta_{2})\|_{\bar{0}}^{2}\right)+\lambda_{2}^{2}\left(\|\nabla(\nabla\cdot u_{2t}(\theta_{1}))\|_{\bar{0}}^{2}+25\|\nabla(\nabla\cdot u_{2t}(\theta_{2}))\|_{\bar{0}}^{2}\right)\right]\right\}
      \end{equation*}
       \begin{equation}\label{49}
       +4\sigma\|e_{h}^{n+1}-e_{h}^{n-1}\|_{\bar{0}}^{2}.
      \end{equation}

       But $u\in [W^{5}_{2}(0,T;\text{\,}H^{5})]^{3}$ implies $u_{4t}\in [W^{1}_{2}(0,T;\text{\,}H^{5})]^{3}$. Further, the first equation in system $(\ref{1})$ shows that $F(u)\in [W^{3}_{2}(0,T;\text{\,}H^{3})]^{3}$. This suggests that $(F(u))_{2t}\in [W^{1}_{2}(0,T;\text{\,}H^{3})]^{3}$. Additionally, $\theta_{1}\in(t_{n-1},\text{\,}t_{n})$, $\theta_{2}\in(t_{n},\text{\,}t_{n+1})$ and $\overline{\epsilon}\in(\min\{\epsilon(t_{n-1}),\epsilon(t_{n+1})\},\text{\,}
      \max\{\epsilon(t_{n-1}),\epsilon(t_{n+1})\})\subset(t_{n-1},\text{\,}t_{n+1})$, for every $n=1,2,...,N-1$. Utilizing this, estimate $(\ref{49})$ implies
       \begin{equation*}
      (\|e_{h}^{n+1}-e_{h}^{n}\|_{\bar{0}}^{2}-\|e_{h}^{n}-e_{h}^{n-1}\|_{\bar{0}}^{2})+\frac{\sigma^{2}}{2\rho}\left[-(\|e_{h}^{n+1}-e_{h}^{n}\|_{B}^{2}-
      \|e_{h}^{n}-e_{h}^{n-1}\|_{B}^{2})+(\|e_{h}^{n+1}\|_{B}^{2}-\|e_{h}^{n}\|_{B}^{2})+\right.
      \end{equation*}
      \begin{equation*}
      \left.(\|e_{h}^{n}\|_{B}^{2}-\|e_{h}^{n-1}\|_{B}^{2})\right]\leq \frac{C_{3}C_{2}\sigma^{3}}{4\rho^{2}}\|e_{h}^{n}\|_{B}^{2}+\frac{2\sigma^{7}}{(48\rho)^{2}}\left\{
      26\||(F(u))_{2t}|\|_{\bar{0},\infty}^{2}+128\rho^{2}\||u_{4t}|\|_{\bar{0},\infty}^{2}+\right.
      \end{equation*}
      \begin{equation*}
      \left.52\left[(\lambda_{1}+\lambda_{2})^{2}\||\nabla\cdot\overline{\nabla}u_{2t}|\|_{\bar{0},\infty}^{2}+\lambda_{2}^{2}\||\nabla(\nabla\cdot u_{2t})|\|_{\bar{0},\infty}^{2}\right]\right\}+4\sigma\|e_{h}^{n+1}-e_{h}^{n-1}\|_{\bar{0}}^{2}.
      \end{equation*}

     Summing up this estimate for $l=1,2,...,n$, observing that
     \begin{equation*}
      \underset{l=1}{\overset{n}\sum}\|e_{h}^{l+1}-e_{h}^{l-1}\|_{\bar{0}}^{2}\leq 2\underset{l=1}{\overset{n}\sum}(\|e_{h}^{l+1}-e_{h}^{l}\|_{\bar{0}}^{2}+
      \|e_{h}^{l}-e_{h}^{l-1}\|_{\bar{0}}^{2})\leq 4\underset{l=0}{\overset{n}\sum}\|e_{h}^{l+1}-e_{h}^{l}\|_{\bar{0}}^{2},
     \end{equation*}
     and rearranging term, this provides
     \begin{equation*}
      \|e_{h}^{n+1}-e_{h}^{n}\|_{\bar{0}}^{2}-\frac{\sigma^{2}}{2\rho}\|e_{h}^{n+1}-e_{h}^{n}\|_{B}^{2}+\frac{\sigma^{2}}{2\rho}\left(\||e_{h}^{n+1}|\|_{B}^{2}+\||e_{h}^{n}|\|_{B}^{2}+
      \||e_{h}^{1}-e_{h}^{0}|\|_{B}^{2}\right)\leq \|e_{h}^{1}-e_{h}^{0}\|_{\bar{0}}^{2}+
     \end{equation*}
     \begin{equation*}
      \frac{\sigma^{2}}{2\rho}\left(\||e_{h}^{1}|\|_{B}^{2}+\||e_{h}^{0}|\|_{B}^{2}\right)+\frac{C_{3}C_{2}\sigma^{3}}{4\rho^{2}}\underset{l=1}{\overset{n}\sum}\|e_{h}^{l}\|_{B}^{2}+
      16\sigma\underset{l=0}{\overset{n}\sum}\|e_{h}^{l+1}-e_{h}^{l}\|_{\bar{0}}^{2}+\frac{2n\sigma^{7}}{(48\rho)^{2}}\left\{26\||(F(u))_{2t}|\|_{\bar{0},\infty}^{2}+\right.
      \end{equation*}
       \begin{equation}\label{50}
      \left.128\rho^{2}\||u_{4t}|\|_{\bar{0},\infty}^{2}+52\left[(\lambda_{1}+\lambda_{2})^{2}\||\nabla\cdot\overline{\nabla}u_{2t}|\|_{\bar{0},\infty}^{2}+
      \lambda_{2}^{2}\||\nabla(\nabla\cdot u_{2t})|\|_{\bar{0},\infty}^{2}\right]\right\}.
      \end{equation}

     Using estimate $(\ref{39})$ along with the time step requirement $(\ref{32})$, it is easy to see that
      \begin{equation*}
      \|e_{h}^{n+1}-e_{h}^{n}\|_{\bar{0}}^{2}-\frac{\sigma^{2}}{2\rho}\|e_{h}^{n+1}-e_{h}^{n}\|_{B}^{2}\geq \|e_{h}^{n+1}-e_{h}^{n}\|_{\bar{0}}^{2}-\frac{\sigma^{2}}{2\rho C_{1}h^{2}}\|e_{h}^{n+1}-e_{h}^{n}\|_{\bar{0}}^{2}\geq \beta_{0}\|e_{h}^{n+1}-e_{h}^{n}\|_{\bar{0}}^{2},
      \end{equation*}
      and $\frac{1}{C_{2}}\|e_{h}^{n}\|_{\bar{0}}^{2}\leq \|e_{h}^{n}\|_{B}^{2}$, where $\beta_{0}=1-\frac{C_{ts}^{2}}{2\rho C_{1}}>0$. Utilizing these facts, inequality $(\ref{50})$ becomes
     \begin{equation*}
      \|e_{h}^{n+1}-e_{h}^{n}\|_{\bar{0}}^{2}+\frac{\sigma^{2}}{2\rho\beta_{0}}\left(\|e_{h}^{n+1}\|_{B}^{2}+\frac{1}{C_{2}}\|e_{h}^{n}\|_{\bar{0}}^{2}+
      \|e_{h}^{1}-e_{h}^{0}\|_{B}^{2}\right)\leq \frac{1}{\beta_{0}}\left[\|e_{h}^{1}-e_{h}^{0}\|_{\bar{0}}^{2}+\frac{\sigma^{2}}{2\rho}\left(\|e_{h}^{1}\|_{B}^{2}+
      \|e_{h}^{0}\|_{B}^{2}\right)\right]+
     \end{equation*}
     \begin{equation*}
      \frac{C_{3}C_{2}\sigma^{3}}{4\beta_{0}\rho^{2}}\underset{l=1}{\overset{n}\sum}\|e_{h}^{l}\|_{B}^{2}+\frac{16\sigma}{\beta_{0}}\underset{l=0}{\overset{n}\sum}
      \|e_{h}^{l+1}-e_{h}^{l}\|_{\bar{0}}^{2}+\frac{2n\sigma^{7}}{(48\rho)^{2}\beta_{0}}\left\{26\||(F(u))_{2t}|\|_{\bar{0},\infty}^{2}+128\rho^{2}\||u_{4t}|\|_{\bar{0},\infty}^{2}+\right.
      \end{equation*}
      \begin{equation}\label{51}
      \left.52\left[(\lambda_{1}+\lambda_{2})^{2}\||\nabla\cdot\overline{\nabla}u_{2t}|\|_{\bar{0},\infty}^{2}+
       \lambda_{2}^{2}\||\nabla(\nabla\cdot u_{2t})|\|_{\bar{0},\infty}^{2}\right]\right\}.
      \end{equation}

      For small values of the time step $\sigma$ that satisfies the restriction $(\ref{32})$, using the initial condition $(\ref{s3})$, i.e., $e_{h}^{0}=0$, and applying the discrete Gronwall inequality, estimate $(\ref{51})$ implies
      \begin{equation*}
      \|e_{h}^{n+1}-e_{h}^{n}\|_{\bar{0}}^{2}+\frac{\sigma^{2}}{2\rho\beta_{0}}\left(\|e_{h}^{n+1}\|_{B}^{2}+\frac{1}{C_{2}}\|e_{h}^{n}\|_{\bar{0}}^{2}\right) \leq \left\{\frac{1}{\beta_{0}}\|e_{h}^{1}\|_{\bar{0}}^{2}+\frac{C_{3}C_{2}\sigma^{3}}{4\beta_{0}\rho^{2}}\underset{l=1}{\overset{n}\sum}\|e_{h}^{l}\|_{B}^{2}+
      \frac{2n\sigma^{7}}{(48\rho)^{2}\beta_{0}}\left\{26\||(F(u))_{2t}|\|_{\bar{0},\infty}^{2}\right.\right.
     \end{equation*}
     \begin{equation*}
      \left.\left.+128\rho^{2}\||u_{4t}|\|_{\bar{0},\infty}^{2}+52\left[(\lambda_{1}+\lambda_{2})^{2}\||\nabla\cdot\overline{\nabla}u_{2t}|\|_{\bar{0},\infty}^{2}+
      \lambda_{2}^{2}\||\nabla(\nabla\cdot u_{2t})|\|_{\bar{0},\infty}^{2}\right]\right\}\right\}e^{\frac{16n\sigma}{\beta_{0}}}.
      \end{equation*}

     Since $n\sigma\leq T$, $\frac{\sigma^{2}}{4C_{2}\rho\beta_{0}}<1$, and $\|e_{h}^{n+1}\|_{\bar{0}}^{2}\leq(\|e_{h}^{n+1}-e_{h}^{n}\|_{\bar{0}}+\|e_{h}^{n}\|_{\bar{0}})^{2}
     \leq 2(\|e_{h}^{n+1}-e_{h}^{n}\|_{\bar{0}}^{2}+\|e_{h}^{n}\|_{\bar{0}}^{2})$. Setting $\beta_{1}=\frac{\sigma^{2}}{4C_{2}\rho\beta_{0}}$, it is easy to see that this estimate becomes
      \begin{equation*}
      \beta_{1}\|e_{h}^{n+1}\|_{\bar{0}}^{2}+2C_{2}\beta_{1}\|e_{h}^{n+1}\|_{B}^{2} \leq \left\{\frac{1}{\beta_{0}}\|e_{h}^{1}\|_{\bar{0}}^{2}+
      \frac{C_{1}C_{2}C_{3}\beta_{1}\sigma}{\rho}\underset{l=1}{\overset{n}\sum}\|e_{h}^{l}\|_{B}^{2}+\frac{C_{1}\beta_{1}\sigma^{4}}{288\rho}
      \left\{26\||(F(u))_{2t}|\|_{\bar{0},\infty}^{2}\right.\right.
     \end{equation*}
     \begin{equation*}
      \left.\left.+128\rho^{2}\||u_{4t}|\|_{\bar{0},\infty}^{2}+52\left[(\lambda_{1}+\lambda_{2})^{2}\||\nabla\cdot\overline{\nabla}u_{2t}|\|_{\bar{0},\infty}^{2}+
      \lambda_{2}^{2}\||\nabla(\nabla\cdot u_{2t})|\|_{\bar{0},\infty}^{2}\right]\right\}\right\}e^{\frac{16T}{\beta_{0}}}.
      \end{equation*}

      Multiplying both sides of this estimate by $\frac{1}{2C_{2}\beta_{1}}$ and applying the Gronwall inequality, this gives
       \begin{equation*}
      \|e_{h}^{n+1}\|_{\bar{0}}^{2}+\|e_{h}^{n+1}\|_{B}^{2}\leq \beta_{2}^{-1}e^{\frac{16T}{\beta_{0}}}\left\{\frac{1}{2C_{2}\beta_{1}\beta_{0}}\|e_{h}^{1}\|_{\bar{0}}^{2}+
      \frac{\sigma^{4}}{576\rho}\left\{26\||(F(u))_{2t}|\|_{\bar{0},\infty}^{2}+128\rho^{2}\||u_{4t}|\|_{\bar{0},\infty}^{2}+\right.\right.
     \end{equation*}
     \begin{equation}\label{52}
      \left.\left.52\left[(\lambda_{1}+\lambda_{2})^{2}\||\nabla\cdot\overline{\nabla}u_{2t}|\|_{\bar{0},\infty}^{2}+\lambda_{2}^{2}
      \||\nabla(\nabla\cdot u_{2t})|\|_{\bar{0},\infty}^{2}\right]\right\}\right\}\exp\left(\frac{C_{1}C_{3}T}{2\rho}e^{\frac{16T}{\beta_{0}}}\right),
      \end{equation}
      where $\beta_{2}=\min\{1,\frac{1}{2C_{2}}\}$.\\

       Let $v_{h}$ be an arbitrary element in $\mathcal{U}_{h}$. So, $\|e_{h}^{1}\|_{\bar{0}}^{2}=\|(u_{h}^{1}-v_{h})+ (v_{h}-u^{1})\|_{\bar{0}}^{2} \leq 2(\|u_{h}^{1}-v_{h}\|_{\bar{0}}^{2}+\|v_{h}-u^{1}\|_{\bar{0}}^{2})$. Since $\underset{v_{h}\in\mathcal{U}_{h}}{\inf}\|u_{h}^{1}-v_{h}\|_{\bar{0}}=0$, taking the infimum over $v_{h}\in\mathcal{U}_{h}$, in both sides of estimate $(\ref{52})$ and utilizing the first inequality in relation $(\ref{33a})$, this yields
       \begin{equation*}
      \|e_{h}^{n+1}\|_{\bar{0}}^{2}+\|e_{h}^{n+1}\|_{B}^{2}\leq \beta_{2}^{-1}e^{\frac{16T}{\beta_{0}}}\left\{\frac{\gamma_{0}^{2}h^{8}}{C_{2}\beta_{1} \beta_{0}} \|u^{1}\|_{[H^{5}]^{3}}^{2}+\frac{\sigma^{4}}{576\rho}\left\{26\||(F(u))_{2t}|\|_{\bar{0},\infty}^{2}+128\rho^{2}\||u_{4t}|\|_{\bar{0},\infty}^{2}+\right.\right.
     \end{equation*}
     \begin{equation}\label{53}
      \left.\left.52\left[(\lambda_{1}+\lambda_{2})^{2}\||\nabla\cdot\overline{\nabla}u_{2t}|\|_{\bar{0},\infty}^{2}+\lambda_{2}^{2} \||\nabla(\nabla\cdot u_{2t})|\|_{\bar{0},\infty}^{2} \right]\right\}\right\}\exp\left(\frac{C_{1}C_{3}T}{2\rho}e^{\frac{16T}{\beta_{0}}}\right).
      \end{equation}

      But, $\beta_{1}=\frac{\sigma^{2}}{4C_{2}\beta_{0}\rho}$, using the time step limitation $(\ref{32})$, it holds: $\frac{\gamma_{0}^{2}h^{8}}{C_{2}\beta_{1} \beta_{0}} =4\gamma_{0}^{2}h^{8}\sigma^{-2}\leq 4C_{ts}^{2}\gamma_{0}^{2}h^{6}$. Substituting this into estimate $(\ref{53})$ and utilizing the norm $\|\cdot\|_{\bar{0},B}$ defined in equation $(\ref{38a})$, to obtain
      \begin{equation*}
      \|e_{h}^{n+1}\|_{\bar{0},B}^{2}\leq \beta_{2}^{-1}e^{\frac{16T}{\beta_{0}}}\left\{4C_{ts}^{2}\gamma_{0}^{2}h^{6}\|u^{1}\|_{[H^{5}]^{3}}^{2} +\frac{\sigma^{4}}{576\rho}\left\{26\||(F(u))_{2t}|\|_{\bar{0},\infty}^{2}+128\rho^{2}\||u_{4t}|\|_{\bar{0},\infty}^{2}+\right.\right.
     \end{equation*}
     \begin{equation*}
      \left.\left.52\left[(\lambda_{1}+\lambda_{2})^{2}\||\nabla\cdot\overline{\nabla}u_{2t}|\|_{\bar{0},\infty}^{2}+\lambda_{2}^{2} \||\nabla(\nabla\cdot u_{2t})|\|_{\bar{0},\infty}^{2} \right]\right\}\right\}\exp\left(\frac{C_{1}C_{3}T}{2\rho}e^{\frac{16T}{\beta_{0}}}\right).
      \end{equation*}

      The square root in both sides of this inequality gives
      \begin{equation*}
      \|e_{h}^{n+1}\|_{\bar{0},B}\leq\beta_{2}^{-\frac{1}{2}}e^{\frac{8T}{\beta_{0}}}\left\{4C_{ts}^{2}\gamma_{0}^{2}h^{6}\|u^{1}\|_{[H^{5}]^{3}}^{2} +\frac{\sigma^{4}}{576\rho}\left\{26\||(F(u))_{2t}|\|_{\bar{0},\infty}^{2}+128\rho^{2}\||u_{4t}|\|_{\bar{0},\infty}^{2}+\right.\right.
     \end{equation*}
     \begin{equation}\label{54}
      \left.\left.52\left[(\lambda_{1}+\lambda_{2})^{2}\||\nabla\cdot\overline{\nabla}u_{2t}|\|_{\bar{0},\infty}^{2}+\lambda_{2}^{2}\||\nabla(\nabla\cdot u_{2t})|\|_{\bar{0},\infty}^{2}\right]\right\}\right\}^{\frac{1}{2}}\exp\left(\frac{C_{1}C_{3}T}{4\rho}e^{\frac{16T}{\beta_{0}}}\right).
      \end{equation}

      It is easy to see that
      \begin{equation*}
      4C_{ts}^{2}\gamma_{0}^{2}h^{6} \|u^{1}\|_{[H^{5}]^{3}}^{2}+\frac{\sigma^{4}}{576\rho}\left\{26\||(F(u))_{2t}|\|_{\bar{0},\infty}^{2}+128\rho^{2}
      \||u_{4t}|\|_{\bar{0},\infty}^{2}+52\left[(\lambda_{1}+\lambda_{2})^{2}\||\nabla\cdot\overline{\nabla}u_{2t}|\|_{\bar{0},\infty}^{2}+\right.\right.
     \end{equation*}
     \begin{equation*}
      \left.\left.\lambda_{2}^{2} \||\nabla(\nabla\cdot u_{2t})|\|_{\bar{0},\infty}^{2} \right]\right\}\leq \left\{4C_{ts}^{2}\gamma_{0}^{2}\|u^{1}\|_{[H^{5}]^{3}}^{2} +\frac{1}{576\rho}\left\{26\||(F(u))_{2t}|\|_{\bar{0},\infty}^{2}+128\rho^{2}\||u_{4t}|\|_{\bar{0},\infty}^{2}+\right.\right.
     \end{equation*}
     \begin{equation*}
      \left.\left.52\left[(\lambda_{1}+\lambda_{2})^{2}\||\nabla\cdot\overline{\nabla}u_{2t}|\|_{\bar{0},\infty}^{2}+\lambda_{2}^{2}\||\nabla(\nabla\cdot u_{2t})|\|_{\bar{0},\infty}^{2} \right]\right\}\right\}(\sigma^{2}+h^{3})^{2}.
     \end{equation*}

     Taking the square root of this estimate and substituting the obtained inequality into estimate $(\ref{54})$ to end the proof of estimate $(\ref{36})$.
      Since $\|u_{h}^{n+1}\|_{\bar{0},B}-\|u^{n+1}\|_{\bar{0},B}\leq\|e_{h}^{n+1}\|_{\bar{0},B}$, estimate $(\ref{54})$ implies
      \begin{equation*}
      \|u_{h}^{n+1}\|_{\bar{0},B} \leq \|u^{n+1}\|_{\bar{0},B}+\beta_{2}^{-\frac{1}{2}}\exp\left(\frac{8T}{\beta_{0}}+\frac{C_{1}C_{3}T}{4\rho}e^{\frac{16T}{\beta_{0}}}\right)
      \left\{4C_{ts}^{2}\gamma_{0}^{2}\|u^{1}\|_{[H^{5}]^{3}}^{2} +\frac{\sigma^{4}}{576\rho}\left\{26\||(F(u))_{2t}|\|_{\bar{0},\infty}^{2}+\right.\right.
     \end{equation*}
     \begin{equation*}
      \left.\left.128\rho^{2}\||u_{4t}|\|_{\bar{0},\infty}^{2}+52\left[(\lambda_{1}+\lambda_{2})^{2}\||\nabla\cdot\overline{\nabla}u_{2t}|\|_{\bar{0},\infty}^{2}+
      \lambda_{2}^{2} \||\nabla(\nabla\cdot u_{2t})|\|_{\bar{0},\infty}^{2} \right]\right\}\right\}^{\frac{1}{2}}(\sigma^{2}+h^{3}),
      \end{equation*}
      for $n=0,1,...,N-1$. This completes the proof of Theorem $\ref{t1}$.
      \end{proof}

      \section{Numerical experiments}\label{sec4}

      \text{\,\,\,\,\,\,\,\,\,\,}This section considers a high-order explicit computational approach for simulating a three-dimensional system of nonlinear elastodynamic sine-Gordon model $(\ref{1})$ with initial and boundary conditions $(\ref{2})$ and $(\ref{3})$, respectively. Two numerical examples are carried out to confirm the theory and to demonstrate the utility and efficiency of the proposed technique $(\ref{s1})$-$(\ref{s5})$. To check the stability and convergence order of the proposed technique $(\ref{s1})$-$(\ref{s5})$, we take $h=\frac{1}{2^{m}},$ for $m=2,3,4,5,6$, where $h=\max\{h_{T},\text{\,\,}T\in\Pi_{h}\}$. Furthermore, we use a uniform time step $\sigma=2^{-m}$, for $m=5,6,7,8,9$. We compute the errors: $u_{h}^{n}-u^{n}$ and $\psi_{h}^{n}-\psi^{n}$ together with the displacement $u_{h}$ and symmetric stress tensor $\psi_{h}$, at time $t_{n}$ using the norms $\||\cdot|\|_{\bar{0},\infty}$ and $\||\cdot|\|_{*,\infty}$, defined as
      \begin{equation*}
       \||v|\|_{\bar{0},\infty}=\underset{0\leq n\leq N}{\max}\|v^{n}\|_{\bar{0}},\text{\,\,\,}\forall v\in\mathcal{U}\text{\,\,\,and\,\,\,}\||\tau|\|_{*,\infty}=\underset{0\leq n\leq N}{\max}\|\tau^{n}\|_{*},\text{\,\,\,}\forall \tau\in\mathcal{M}.
         \end{equation*}

        It's worth mentioning that determining the exact solution of the initial-boundary value problem $(\ref{1})$-$(\ref{3})$ is too difficult and sometimes impossible. Hence, the spatial errors are calculated assuming that the exact solution is the approximate one obtained with the time step $\sigma=2^{-7}$, whereas the temporal errors are computed assuming that the analytical solution equals to the numerical one with the time step $\sigma=2^{-10}$. Additionally, the space convergence order, $CO(h)$, of the constructed computational scheme is calculated utilizing the formula
         \begin{equation*}
          CO(h)=\frac{\log\left(\frac{\||u_{2h}-u|\|_{\bar{0},\infty}}{\||u_{h}-u|\|_{\bar{0},\infty}}\right)}{\log(2)},\text{\,\,}
          \frac{\log\left(\frac{\||\tau_{2h}-\tau|\|_{*,\infty}}{\||\tau_{h}-\tau|\|_{*,\infty}}\right)}{\log(2)},
         \end{equation*}
          where, $z_{2h}$ and $z_{h}$ are the spatial errors associated with the grid sizes $2h$ and $h$, respectively, while the convergence order in time, $CO(\sigma)$, is estimated using the formula
         \begin{equation*}
          CO(\sigma)=\frac{\log\left(\frac{\||u_{2\sigma}-u|\|_{\bar{0},\infty}}{\||u_{\sigma}-u|\|_{\bar{0},\infty}}\right)}{\log(2)},\text{\,\,}
          \frac{\log\left(\frac{\||\tau_{2\sigma}-\tau|\|_{*,\infty}}{\||\tau_{\sigma}-\tau|\|_{*,\infty}}\right)}{\log(2)},
         \end{equation*}
         where $z_{\sigma}$ and $z_{2\sigma}$ denote the errors in time corresponding to time steps $\sigma$ and $2\sigma$, respectively. It's worth mentioning that the numerical computations are carried out using MATLAB R$2007b$.\\

          $\bullet$ \textbf{Example 1}. We consider the three-dimensional system of nonlinear elastodynamic sine-Gordon equations $(\ref{1})$ defined on the domain $\overline{\Omega}=[0,\text{\,}1]^{3}$. The final time $T=1$. The physical parameters are given by: $\rho=1$, $\beta=0.25$, $E=2.5$, $\lambda_{1}=\frac{\beta E}{(1+\beta)(1-2\beta)}=1$ and $\lambda_{2}=\frac{E}{2(1+\beta)}=1$. The nonlinear source term $F(u(x,t))=[\sin u_{1}(x,t),\sin u_{2}(x,t),\sin u_{3}(x,t)]^{T}$. The initial conditions are given by equations $(\ref{2})$ and $(\ref{4})$ while the boundary conditions are defined as: $u(t)|_{\Gamma}=\vec{0}$, for every $t\in[0,\text{\,}1]$. The disaster starts at the focus $B(\bar{x},\text{\,}r_{0})$, where $\bar{x}=(0.5,0.5,0.5)^{t}$ and $r_{0}=2^{-3}$. We set $C_{p}=\frac{1}{3}$ and $C_{ts}=0.5$ be the positive constants defined in estimates $(\ref{31})$ and $(\ref{32})$, respectively.\\
          
         \textbf{Table 1.} $\label{T1}$ Convergence order $CO(h)$ of the constructed high-order explicit computational approach $(\ref{s1})$-$(\ref{s5})$ utilizing the time step $\sigma=2^{-7}$ and various space steps $h$, satisfying condition $(\ref{32})$.
          \begin{equation*}
          \begin{array}{c c}
          \text{\,developed computational technique,\,\,where\,\,}\sigma=2^{-7}& \\
           \begin{tabular}{ccccccc}
            \hline
            $h$ &  $\||u_{h}-u|\|_{\bar{0},\infty}$ & $CO(h)$ & CPU(s) & $\||\psi_{h}-\psi|\|_{*,\infty}$ & $CO(h)$ & CPU(s)\\
             \hline
            $2^{-2}$ &  $6.1879\times10^{-3}$ & .... & 1.5573    & $3.1009\times10^{-3}$ & .... & 5.0019\\

            $2^{-3}$ &  $8.3599\times10^{-4}$ & 2.8879 &  4.2476 & $3.8571\times10^{-4}$ & 3.0071 & 10.5810\\

            $2^{-4}$ &  $1.0353\times10^{-4}$ & 3.0135 &  10.1339 & $4.5232\times10^{-5}$ & 3.0921 & 26.7900\\

            $2^{-5}$ &  $1.2963\times10^{-5}$ & 2.9976 & 20.5032 & $5.7393\times10^{-6}$ & 2.9784 & 73.8871 \\

            $2^{-6}$ &  $1.5117\times10^{-6}$ & 3.1002 & 65.9342 & $6.5915\times10^{-6}$ & 3.1222 & 220.6784 \\
            \hline
          \end{tabular} &
          \end{array}
          \end{equation*}
          
           \textbf{Table 2.} $\label{T2}$ Convergence order $CO(\sigma)$ of the developed approach $(\ref{s1})$-$(\ref{s5})$ with space step $h=2^{-3}$ and varying time steps $\sigma$, satisfying requirement $(\ref{32})$.
           \begin{equation*}
          \begin{array}{c c}
          \text{\,developed computational approach\,\,where\,\,}h=2^{-3}& \\
           \begin{tabular}{ccccccc}
            \hline
            $\sigma$ &  $\||u_{h}-u|\|_{\bar{0},\infty}$ & $CO(\sigma)$ & CPU(s) & $\||\psi_{h}-\psi|\|_{*,\infty}$ & $CO(\sigma)$ & CPU(s)\\
             \hline
            $2^{-5}$ &  $3.9776\times10^{-2}$ & ....   & 2.1103  & $7.2589\times10^{-3}$ & .... & 2.7398\\

            $2^{-6}$ &  $9.8828\times10^{-3}$ & 2.0089 & 4.2179  & $1.9454\times10^{-3}$ & 1.8997 & 5.4793\\

            $2^{-7}$ &  $2.3024\times10^{-3}$ & 2.1018 & 8.9389  & $4.8885\times10^{-4}$ & 1.9926 & 14.5687\\

            $2^{-8}$ &  $5.7648\times10^{-4}$ & 1.9978 & 21.0940 & $1.2177\times10^{-4}$ & 2.0052 & 38.7366\\

            $2^{-9}$ &  $1.3267\times10^{-4}$ & 2.1194 & 60.5399 & $2.7941\times10^{-5}$ & 2.1237 & 116.1673\\
            \hline
          \end{tabular} &
          \end{array}
          \end{equation*}

            \textbf{Table 3.} $\label{T3}$ Displacement $u_{h}$ and symmetric stress tensor $\psi_{h}=[\psi_{h,ij}]$, for $1\leq i,j\leq3$, provided by the new algorithm using a space step $h=2^{-3}$ and time step $\sigma=2^{-7}$.
          \begin{equation*}
          \begin{array}{c }
          \text{\,Developed approach,\,\,where\,\,}\sigma=2^{-7} \\
           \begin{tabular}{cccccccc}
            \hline
            $h$ &  $\||u_{h}|\|_{\bar{0},\infty}$ & $\||\psi_{h,11}|\|_{0,\infty}$ & $\||\psi_{h,22}|\|_{0,\infty}$ & $\||\psi_{h,33}|\|_{0,\infty,}$ & $\||\psi_{h,12}|\|_{0,\infty}$ & $\||\psi_{h,13}|\|_{0,\infty}$ & $\||\psi_{h,23}|\|_{0,\infty}$\\
             \hline
            $2^{-3}$ & $2190$ & $4811$ & $5757$ & $16313$ & $495$ & $5184$ & $5662$\\
            \hline
          \end{tabular}
          \end{array}
          \end{equation*}
          
           Additionally, the stress tensor is defined as
          \begin{equation*}
           \psi_{h}=\begin{bmatrix}
              4811 & 495 & 5184 \\
              495  & 5757 & 5662 \\
              5184 & 5662 & 16313 \\
            \end{bmatrix}
          \end{equation*}

          $\bullet$ \textbf{Example 2}. Consider the three-dimensional system of nonlinear tectonic deformation problem $(\ref{1})$ defined on the region $\overline{\Omega}\times[0,\text{\,}T]=[-1,\text{\,}1]^{3}\times[0,\text{\,}2]$. The physical parameters are given as: $\rho=1$, $\beta=0.1$, $E=2.5$, $\lambda_{1}=\frac{\beta E}{(1+\alpha)(1-2\beta)}=0.2841$ and $\lambda_{2}=\frac{E}{2(1+\beta)}=1.1364$. The nonlinear source term $F(u(x,t))=[\sin u_{1}(x,t),\sin u_{2}(x,t),\sin u_{3}(x,t)]^{T}$. The initial conditions are given by equations $(\ref{2})$ and $(\ref{4})$ while the boundary conditions are defined as: $u(t)|_{\Gamma}=\vec{0}$, for every $t\in[0,\text{\,}1]$. The disaster is assumed to start at the focus $B(\bar{x},\text{\,}r_{0})$, where $\bar{x}=(0,0,0)^{T}$ and $r_{0}=2^{-3}$. We set $C_{p}=\frac{1}{3}$ and $C_{ts}=0.5$ be the positive constants defined in estimates $(\ref{31})$ and $(\ref{32})$, respectively.\\
          
         \textbf{Table 4.} $\label{T4}$ Convergence order $CO(h)$ of the constructed high-order explicit computational approach $(\ref{s1})$-$(\ref{s5})$ using the time step $\sigma=2^{-7}$ and various space steps $h$, satisfying condition $(\ref{32})$.
          \begin{equation*}
          \begin{array}{c c}
          \text{\,developed computational technique,\,\,where\,\,}\sigma=2^{-7}& \\
           \begin{tabular}{ccccccc}
            \hline
            $h$ &  $\||u_{h}-u|\|_{\bar{0},\infty}$ & $CO(h)$ & CPU(s) & $\||\psi_{h}-\psi|\|_{*,\infty}$ & $CO(h)$ & CPU(s)\\
             \hline
            $2^{-2}$ &  $4.2568\times10^{-3}$ & ....  & 1.8726    & $2.8973\times10^{-3}$ & .... & 4.2341\\

            $2^{-3}$ &  $5.3859\times10^{-4}$ & 2.9825 & 4.1489 & $3.5834\times10^{-4}$ & 3.0153 & 9.5534\\

            $2^{-4}$ &  $6.7319\times10^{-5}$ & 3.0001 & 10.2471 & $4.5157\times10^{-5}$ & 2.9883 & 23.4889\\

            $2^{-5}$ &  $7.9406\times10^{-6}$ & 3.0837 & 28.5267 & $5.6154\times10^{-6}$ &  3.0075 & 64.8248 \\

            $2^{-6}$ &  $9.6852\times10^{-7}$ & 3.0354 & 82.5335 & $6.5451\times10^{-7}$ &  3.1009 & 193.2621 \\
            \hline
          \end{tabular} &
          \end{array}
          \end{equation*}

          \textbf{Table 5.} $\label{T5}$ Convergence order $CO(\sigma)$ of the developed approach $(\ref{s1})$-$(\ref{s5})$ with space step $h=2^{-3}$ and varying time steps $\sigma$, satisfying requirement $(\ref{32})$.
           \begin{equation*}
          \begin{array}{c c}
          \text{\,developed computational approach\,\,where\,\,}h=2^{-3}& \\
           \begin{tabular}{ccccccc}
            \hline
            $\sigma$ &  $\||u_{h}-u|\|_{\bar{0},\infty}$ & $CO(\sigma)$ & CPU(s) & $\||\psi_{h}-\psi|\|_{*,\infty}$ & $CO(\sigma)$ & CPU(s)\\
             \hline
            $2^{-5}$ &  $2.6283\times10^{-3}$ & ....  & 2.1005    & $3.9057\times10^{-4}$ & .... & 4.9216\\

            $2^{-6}$ &  $6.5448\times10^{-4}$ & 2.0057 & 4.6959  & $9.7643\times10^{-5}$ & 2.0000 & 12.0884\\

            $2^{-7}$ &  $1.5559\times10^{-4}$ & 2.0726 & 11.6782  & $2.4545\times10^{-5}$ & 1.9921 & 33.3617\\

            $2^{-8}$ &  $3.6273\times10^{-5}$ & 2.1008 & 31.0616 & $6.0033\times10^{-6}$ & 2.0316 & 96.4586\\

            $2^{-9}$ &  $8.1564\times10^{-6}$ & 2.1529 & 92.7655 & $1.3212\times10^{-6}$ & 2.1839 & 289.3661\\
            \hline
          \end{tabular} &
          \end{array}
          \end{equation*}

           \textbf{Table 6.} $\label{T6}$ Displacement $u_{h}$ and symmetric stress tensor $\psi_{h}=[\psi_{h,ij}]$, for $1\leq i,j\leq3$, provided by the proposed computational technique using a space step $h=2^{-3}$ and time step $\sigma=2^{-7}$.
          \begin{equation*}
          \begin{array}{c }
          \text{\,Developed approach,\,\,where\,\,}\sigma=2^{-7} \\
           \begin{tabular}{cccccccc}
            \hline
            $h$ &  $\||u_{h}|\|_{\bar{0},\infty}$ & $\||\psi_{h,11}|\|_{0,\infty}$ & $\||\psi_{h,22}|\|_{0,\infty}$ & $\||\psi_{h,33}|\|_{0,\infty,}$ & $\||\psi_{h,12}|\|_{0,\infty}$ & $\||\psi_{h,13}|\|_{0,\infty}$ & $\||\psi_{h,23}|\|_{0,\infty}$\\
             \hline
            $2^{-3}$ & $6.283\times10^{7}$ & $2.412\times10^{7}$ & $3.191\times10^{7}$ & $2.6867\times10^{8}$ & $1.66\times10^{6}$ & $ 2.903\times10^{7}$ & $3.011\times10^{7}$\\
            \hline
          \end{tabular}
          \end{array}
          \end{equation*}
           In addition, the symmetric stress tensor is given by
          \begin{equation*}
           \psi_{h}=10^{7}\times\begin{bmatrix}
              2.412 & 0.166 & 2.903 \\
              0.166 & 3.191 & 3.011 \\
              2.903 & 3.011 & 26.867 \\
            \end{bmatrix}
          \end{equation*}

          It follows from \textbf{Tables 1-2 $\&$ 4-5} that the errors associated with the displacement and stress tensor are temporal second-order and spatial third-order. This suggests that the developed algorithm $(\ref{s1})$-$(\ref{s5})$ is temporal second-order accurate and third-order convergent in space. Additionally, \textbf{Tables 3 $\&$ 6} give the displacement ($u_{h}$) and symmetric stress tensor ($\psi_{h}$). It follows from the tables that the approximate solutions do not increase with time and converge to the analytical one.\\

         Furthermore, both displacement and symmetric stress tensor provided by the proposed computational approach $(\ref{s1})$-$(\ref{s5})$ are displayed in Figures $\ref{fig2}$-$\ref{fig3}$. A time step $\sigma=2^{-7}$, and space step $h=2^{-3},$ satisfying the time step requirement $(\ref{32})$ are used. The figures show that the displacement and stress tensor propagate with almost a perfectly value at different positions. Thus, the computed solutions could not increase with time. Additionally, both tables and figures indicate that the numerical solutions do not increase with time and converge to the exact solutions. Moreover, they show that the proposed technique $(\ref{s1})$-$(\ref{s5})$ is not unconditionally unstable, but stability depends on the parameters $\sigma$ and $h$. Finally, the information provided by the tables and figures indicate that the new algorithm $(\ref{s1})$-$(\ref{s5})$ should be considered as a powerful tool to predict runout and delineate the hazardous natural disaster areas.

       \section{General conclusions and future works}\label{sec5}
        This paper has developed a high-order explicit computational technique for solving a three-dimensional system of nonlinear elastodynamic sine-Gordon model $(\ref{1})$, subjects to initial-boundary conditions $(\ref{2})$-$(\ref{3})$. Under an appropriate time step restriction $(\ref{32})$, both stability and error estimates of the developed approach have been deeply analyzed using a constructed strong norm. The theory indicated that the proposed computational technique is stable, temporal second-order accurate and spatial third-order convergent. Utilizing the $L^{\infty}(0,T;\text{\,}L^{2})$-norm, some numerical examples have been carried out to confirm the theoretical results. Specifically, the graphs (Figures $\ref{fig2}$-$\ref{fig3}$) indicate that the constructed numerical approach $(\ref{s1})$-$(\ref{s5})$ is stable whereas \textbf{Tables 1-2 $\&$ 4-5} suggest that the new method is second-order accurate in time and third-order convergent in space. Furthermore, both figures and tables showed that the displacement and stress tensor do not increase with time. This observation suggests that the proposed approach overcomes the limitations raised by a wide set of numerical schemes deeply discussed in the literature \cite{8dg,dg,jtfyl,14dg}. Finally, the tables and figures indicate that the new algorithm should be considered as a powerful tool for providing useful data and related information on some natural disasters which would help communities to be informed about the hazard zones. Our future works will develop a high-order unconditionally stable explicit/implicit approach with finite element method for solving a three-dimensional system of nonlinear tectonic deformation.

      \subsection*{Ethical Approval}
     Not applicable.
     \subsection*{Availability of supporting data}
     Not applicable.
     \subsection*{Declaration of Interest Statement}
     The author declares that he has no conflict of interests.
     \subsection*{Acknowledgment}
     This work was supported and funded by the Deanship of Scientific Research at Imam Mohammad Ibn Saud Islamic University (IMSIU) (grant number IMSIU-DDRSP2501).
     \subsection*{Authors' contributions}
     The whole work has been carried out by the author.

     \begin{figure}
         \begin{center}
         Stability analysis of the developed computational approach for three-dimensional system of elastodynamic sine-Gordon model.
         \begin{tabular}{c c}
         \psfig{file=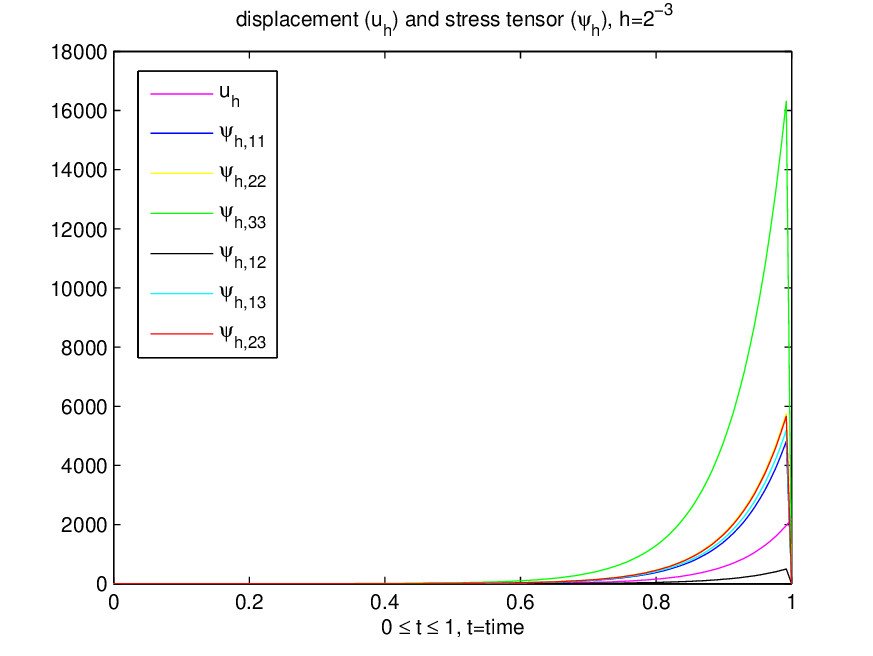,width=7cm} & \psfig{file=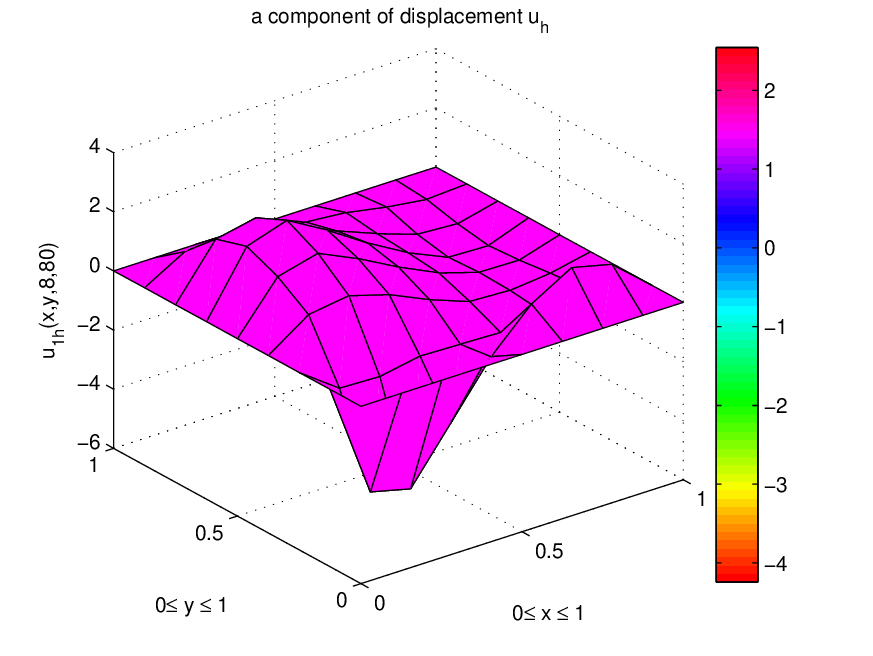,width=7cm}\\
         \psfig{file=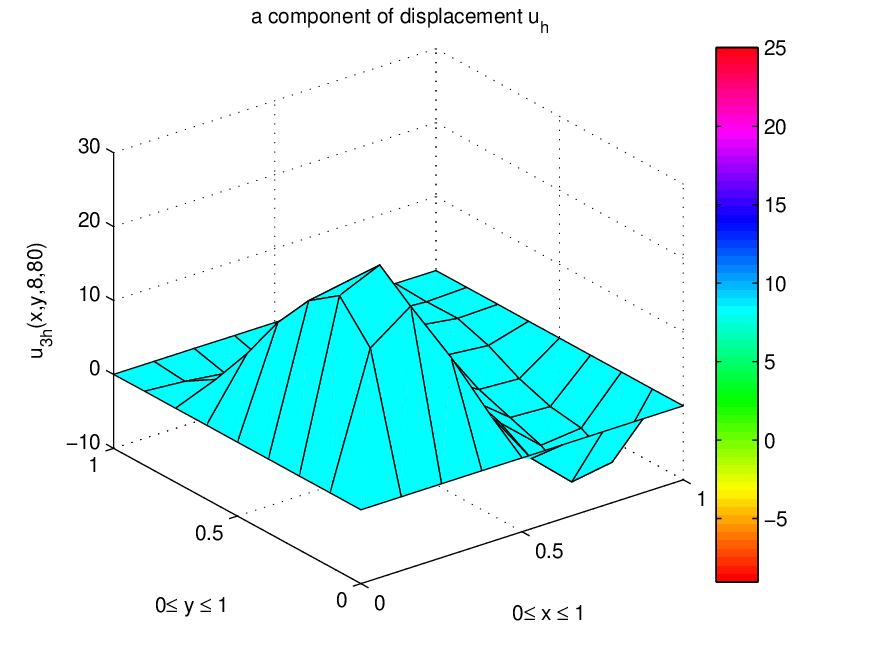,width=7cm} & \psfig{file=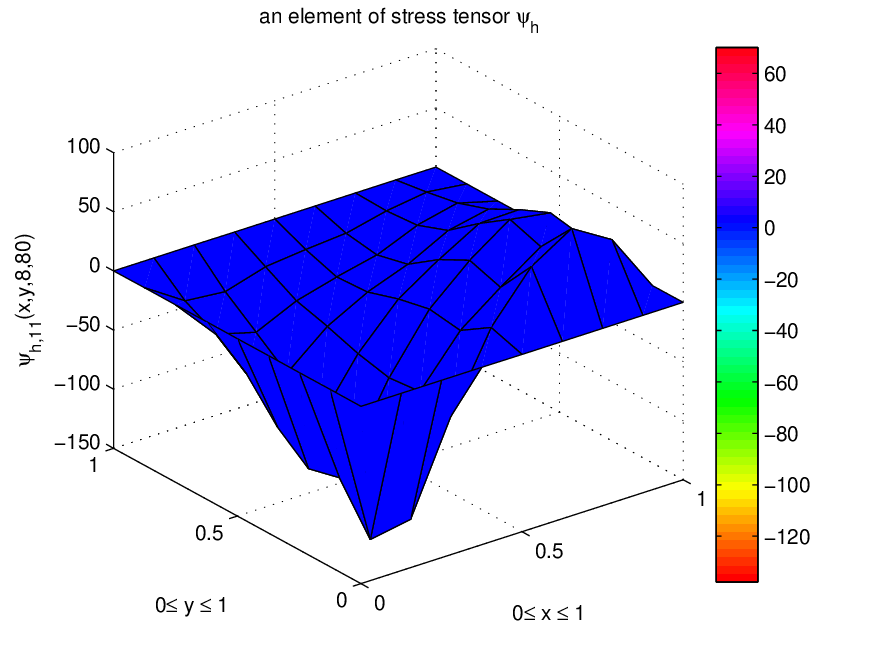,width=7cm}\\
         \psfig{file=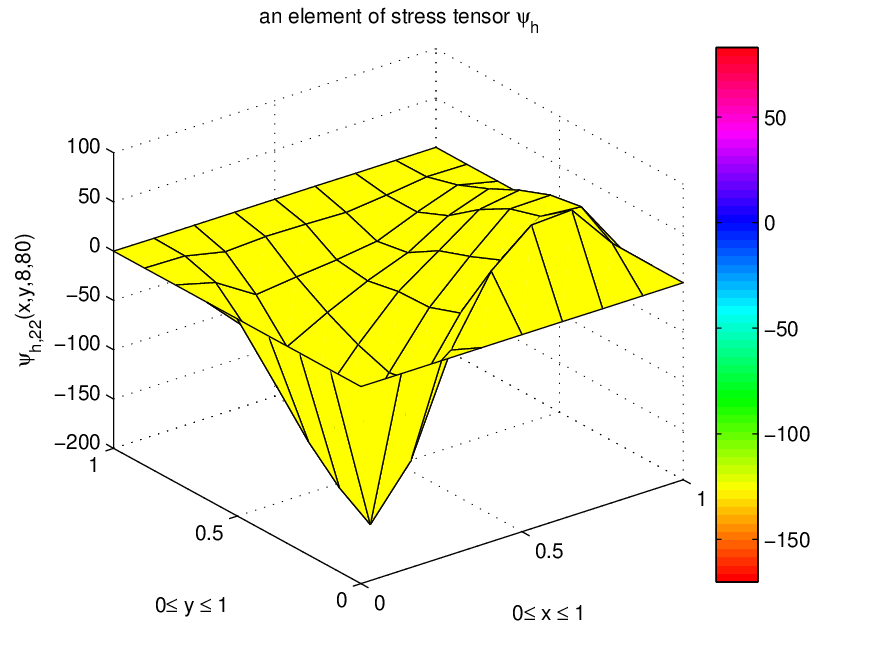,width=7cm} & \psfig{file=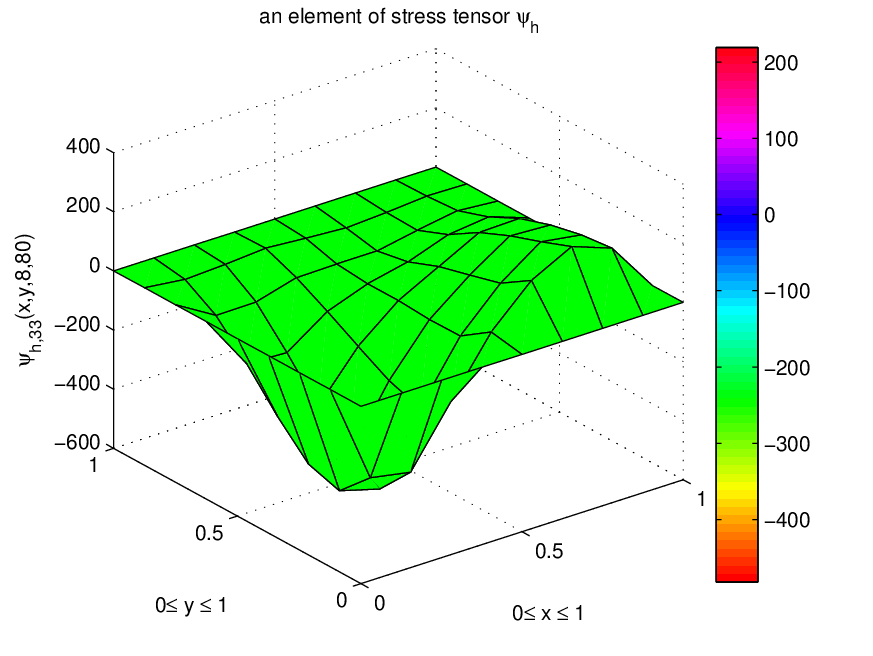,width=7cm}\\
         \psfig{file=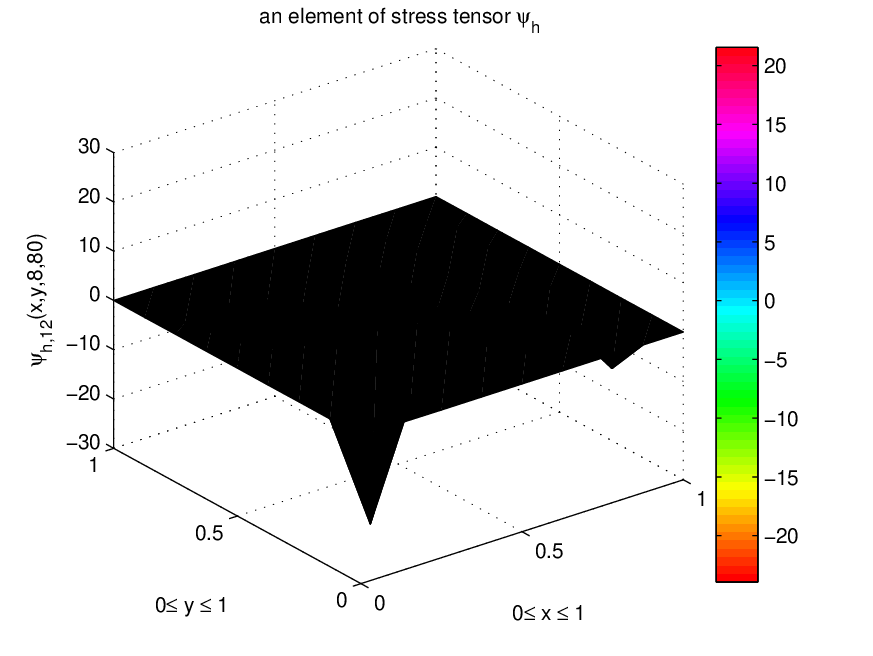,width=7cm} & \psfig{file=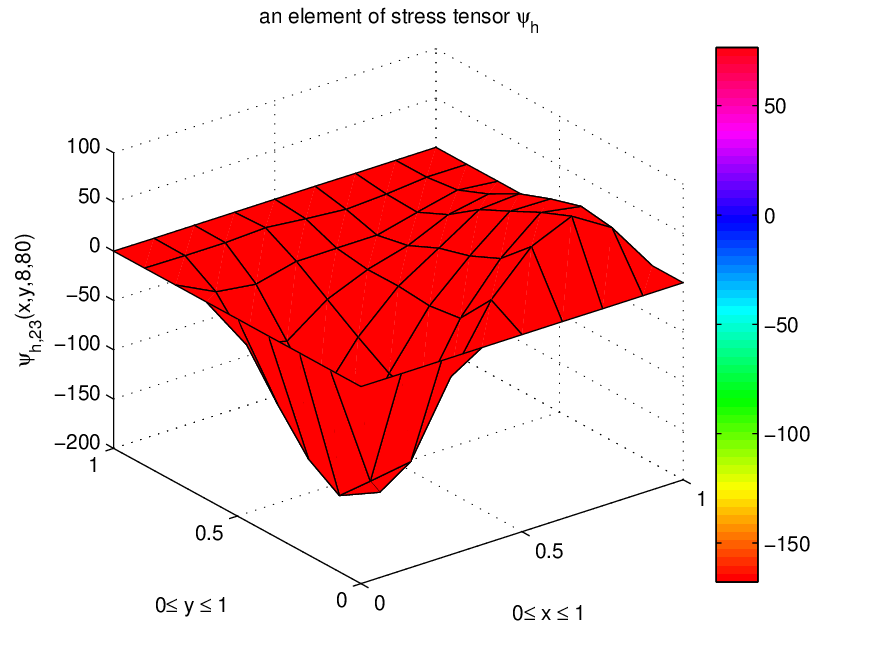,width=7cm}
         \end{tabular}
        \end{center}
        \caption{Graphs of displacement $(u_{h})$ and stress tensor $(\psi_{h})$ corresponding to Example 1.}
        \label{fig2}
        \end{figure}

       \begin{figure}
       \begin{center}
       Stability analysis of the developed computational approach for three-dimensional system of elastodynamic sine-Gordon model.
       \begin{tabular}{c c}
         \psfig{file=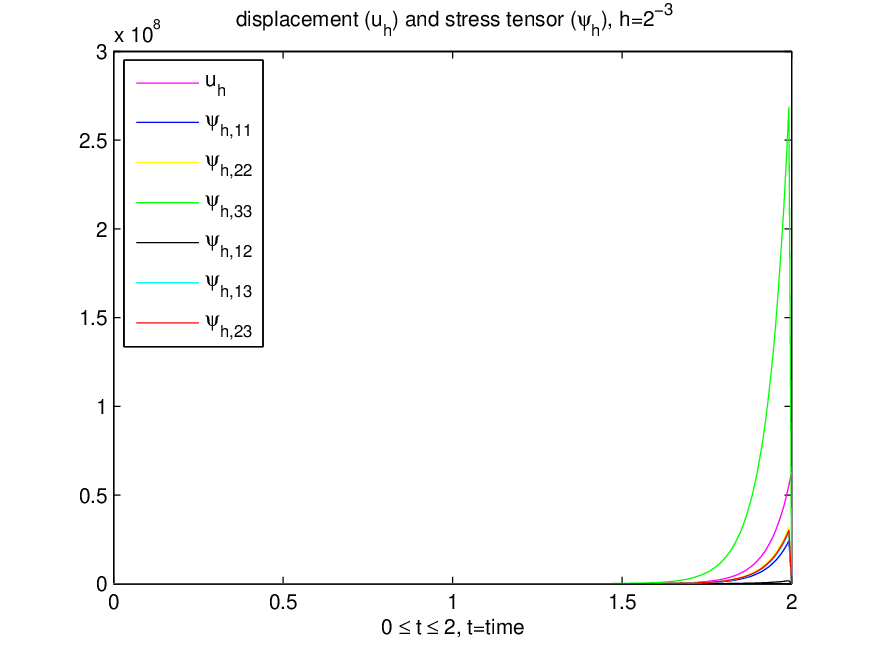,width=7cm} & \psfig{file=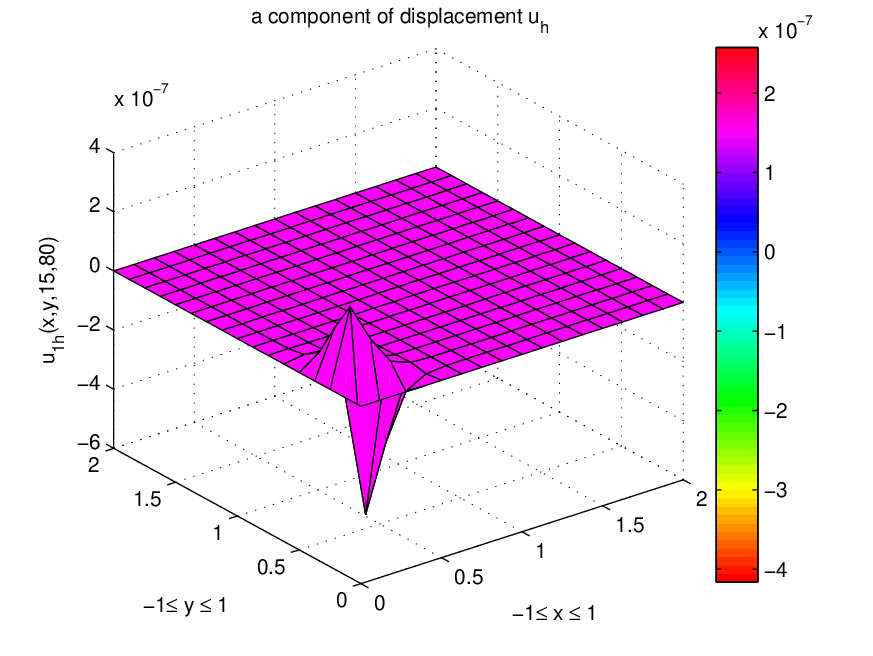,width=7cm}\\
         \psfig{file=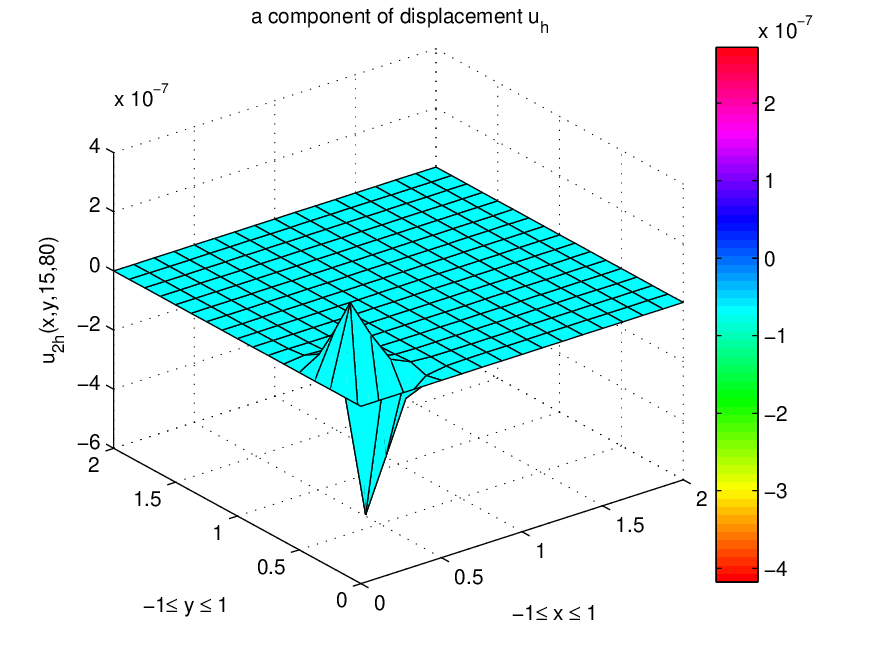,width=7cm} & \psfig{file=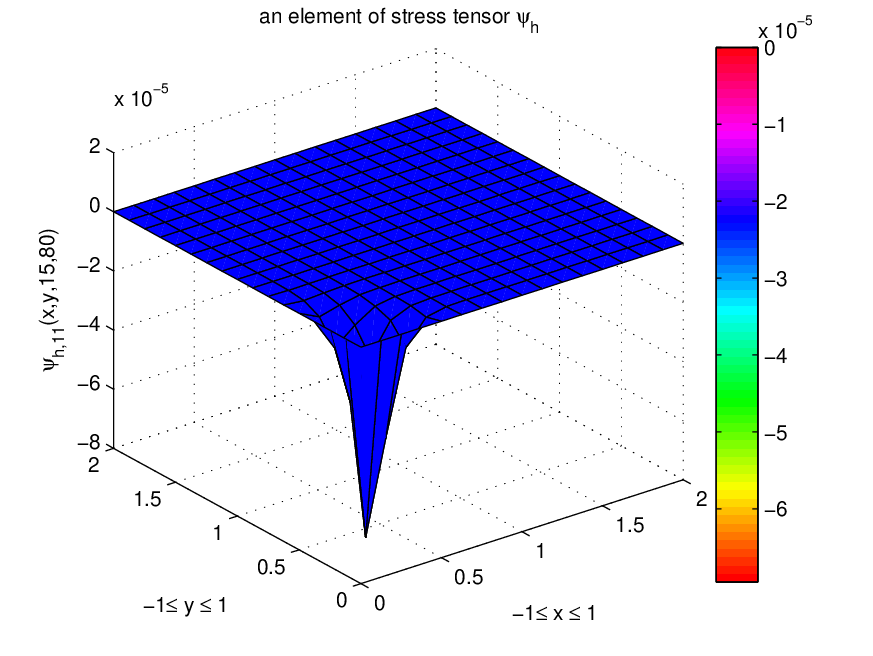,width=7cm}\\
         \psfig{file=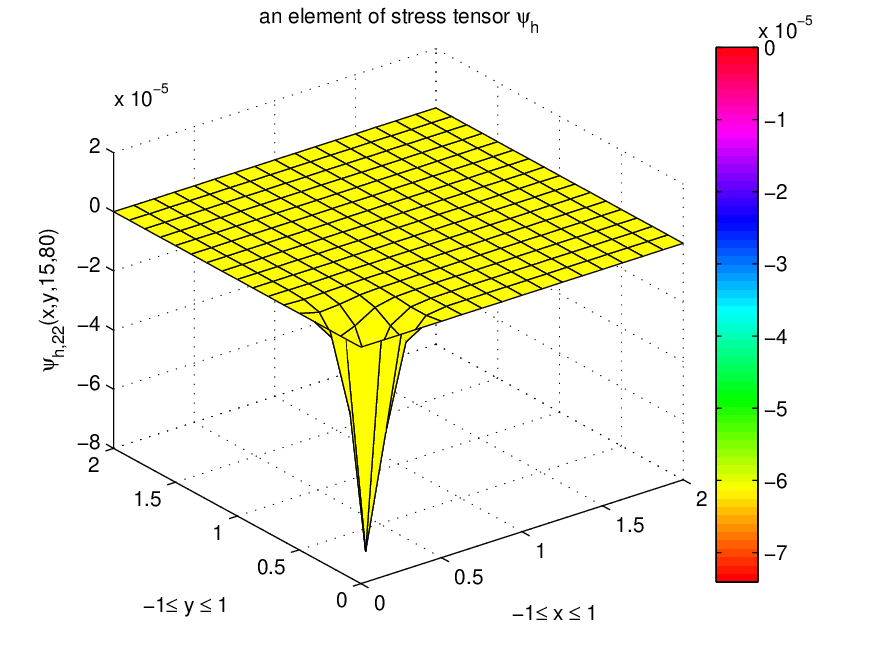,width=7cm} & \psfig{file=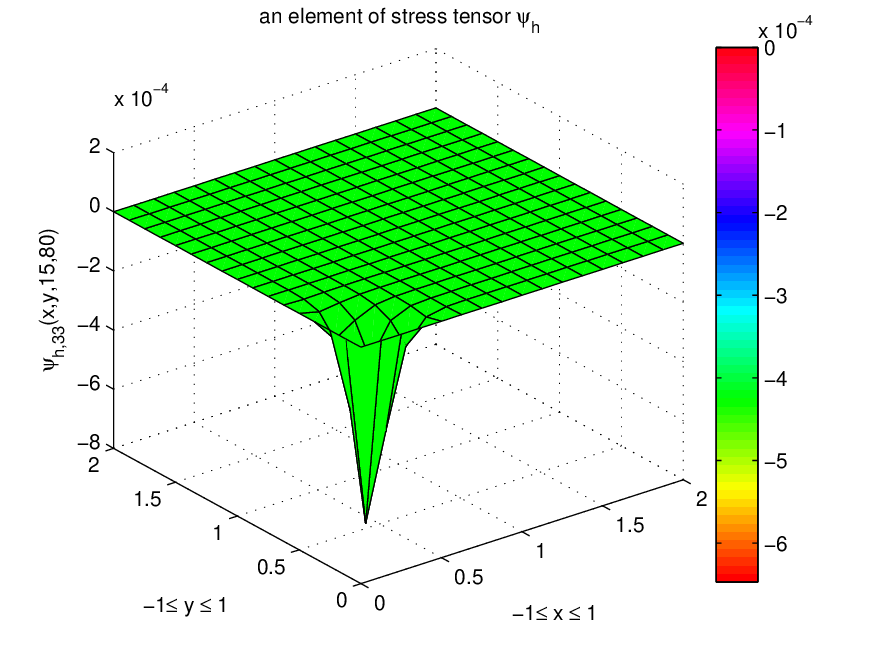,width=7cm}\\
         \psfig{file=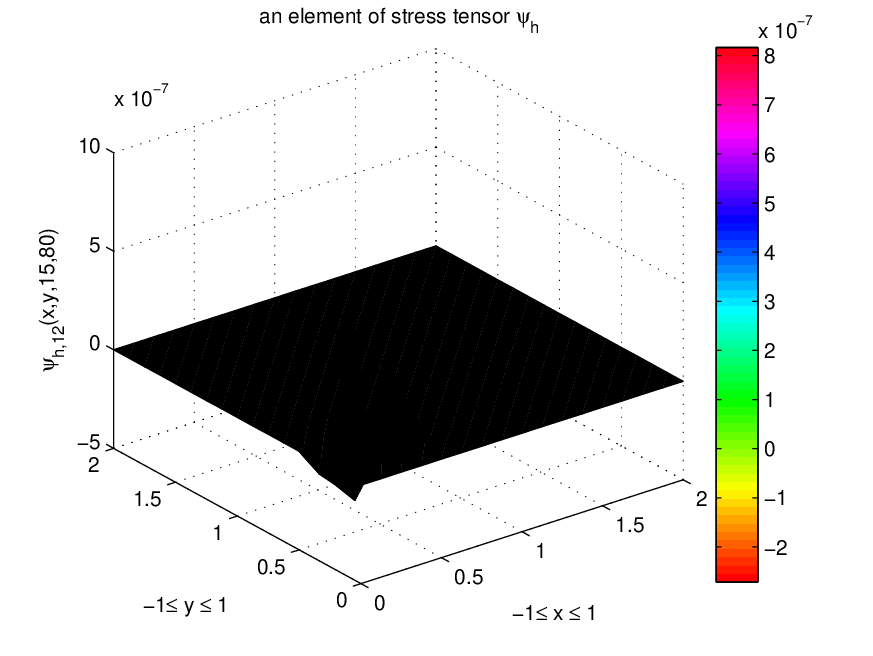,width=7cm} & \psfig{file=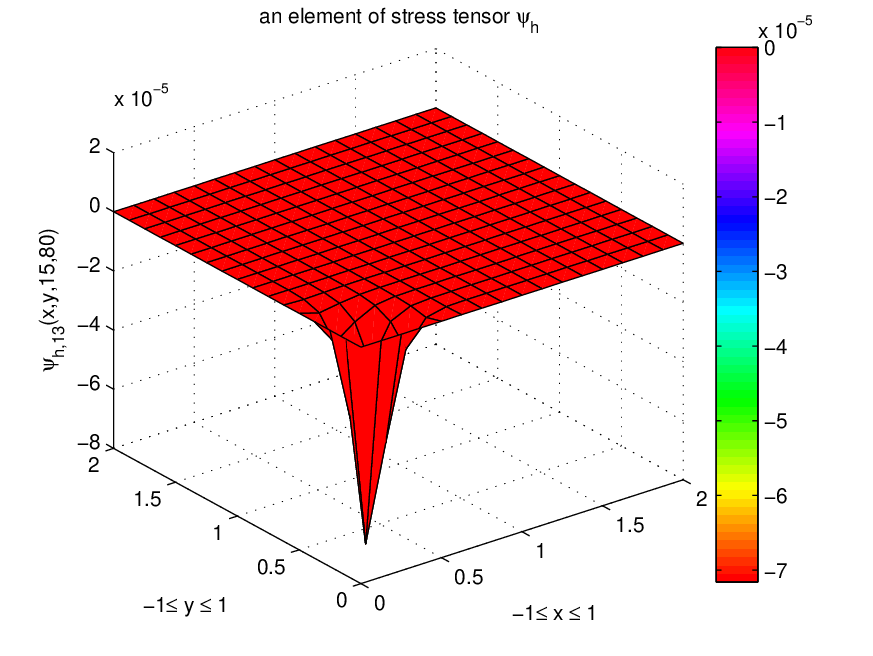,width=7cm}
         \end{tabular}
        \end{center}
         \caption{Graphs of displacement $(u_{h})$ and stress tensor $(\psi_{h})$ corresponding to Example 2.}
          \label{fig3}
          \end{figure}

     \end{document}